\documentclass[10pt]{article}
\usepackage{amsmath}
\usepackage{amsfonts}
\usepackage{amssymb}
\usepackage{amsmath}
\usepackage{amsthm}
\usepackage{latexsym}
\usepackage{graphicx}
\usepackage{epstopdf}
\usepackage{color}
\usepackage[english]{babel}
\usepackage{float}

\newtheorem{theorem}{Theorem}[section]

\newtheorem{definition}[theorem]{Definition}

\newtheorem{lemma}[theorem]{Lemma}
\newtheorem{proposition}[theorem]{Proposition}

\begin{document}
\title{The topology of a subspace of the Legendrian curves on a closed contact 3-manifold}

\author{Ali Maalaoui$^{(1)}$ \& Vittorio Martino$^{(2)}$}
\addtocounter{footnote}{1}
\footnotetext{Department of Mathematics,
Rutgers University - Hill Center for the Mathematical Sciences
110 Frelinghuysen Rd., Piscataway 08854-8019 NJ, USA. E-mail address:
{\tt{maalaoui@math.rutgers.edu}}}
\addtocounter{footnote}{1}
\footnotetext{SISSA, International School for Advanced Studies,
via Bonomea 265, 34136 Trieste, Italy. E-mail address:
{\tt{vmartino@sissa.it}}}
\date{}
\maketitle

{\noindent\bf Abstract} {\small In this paper we study a subspace of the space of Legendrian loops and we show that the injection of this space into the full loop space is an $S^{1}$-equivariant homotopy equivalence. This space can be also seen as the space of zero Maslov index Legendrian loops and it shows up as a suitable space of variations in contact form geometry.}

\section{Introduction}

\noindent
Let $M$ be a 3-dimensional smooth closed (compact, $\partial M=\emptyset$) and orientable manifold, and let $\alpha$ be a 1-form on it. The couple $(M,\alpha)$ is said to be a contact manifold if the form $\alpha \wedge d\alpha$ is a volume form on $M$. A curve $x\in \mathcal{H}^{1}(S^{1},M)$ is said to be legendrian if its tangent vector is in the kernel of $\alpha$, that is $\alpha(\dot{x})=0$. Hence we let $\mathcal{L}_{\alpha}$ denote the space of legendrian closed curves on $M$. This space is a subset of the free loop space of $M$ denoted by $\Lambda(S^1,M)$.
Now we recall a result of Smale \cite{Smale}:\\

\noindent
{\textbf{Theorem}} (Smale).  \emph{Let $(M,\alpha)$ be a contact manifold, then the injection
$$j:\mathcal{L}_{\alpha}\hookrightarrow \Lambda(S^1,M)$$
is an $S^1$-equivariant homotopy equivalence.}\\

\noindent
In this paper we are going to prove a theorem that can be seen as related to the above theorem: the framework will be slightly  different and the space $\mathcal{L}_{\alpha}$ will be replaced by a smaller space $\mathcal{C}_{\beta}$, that appears to be convenient in some variational problems in contact form geometry (see for instance \cite{B1},\cite{B2} and \cite{B4}). We will introduce the following assumption:
$$
\begin{array}{ll}
(A) &  \mbox{there exists a smooth vector field $v\in \ker (\alpha)$ such that the dual}\\
    &  \mbox{1-form $\beta=d\alpha(v,\cdot)$ is a contact form with the same orientation than $\alpha$.}
\end{array}
$$
Under $(A)$, we renormalize $v$ onto $\lambda v$ so that $\alpha \wedge d\alpha=\beta \wedge d\beta$. \\
By Smale's theorem, we know that the injection $\mathcal{L}_{\beta}$ in $\Lambda(S^{1},M)$ is an $S^1$-equivariant homotopy equivalence. We are interested in a space that is smaller than $\mathcal{L}_{\beta}$ and it is defined in the following way:\\
Let
$$\mathcal{C}_{\beta}=\left \{ x\in \mathcal{L}_{\beta}; \alpha_{x}(\dot{x})=c>0 \right \}$$
where $c$ is a constant that varies with the curve $x$.\\
\noindent
The space $\mathcal{C}_{\beta}$ is very useful in contact geometry and it is of independent interest in differential topology.
For example, let us take the framework of $(S^3,\alpha_{0})$, the standard contact form on $S^3$, and let
$$v=-x_{4}\partial_{x_1}-x_3 \partial_{x_2}+x_2\partial_{x_3}+x_1\partial_{x_4}$$
be a Hopf fibration vector field in $\ker \alpha_{0}$. The space $\mathcal{C}_{\beta}$ can be identified as the lift to $S^3$ (according to some rules, see \cite{B1}) of the space $Imm_{0}(S^1,S^2)$ of immersed curves from $S^{1}$ into $S^2$ of Maslov index zero. Smale's theorem \cite{Smale} asserts then that the injection $\mathcal{C}_{\beta}\hookrightarrow \Lambda(S^1,S^3)$ is an $S^1$-equivariant homotopy equivalence.\\
In this paper, we extend this result to a more general framework of $(M,\alpha)$ under (A) and an additional assumption that we introduce below.
We need, in order to state this second assumption, to introduce the one-parameter group generated by $v$ that we will denote by $\varphi_s$.\\
From \cite{B1} and \cite{B4} we know that the kernel of a contact form rotates monotonically in a frame transported by $\varphi_s$ along $v$. Based on this fact we give the following definition.\\
\begin{definition}
We say that $\ker\alpha$ turns well along $v$, if starting from any $x_{0}$ in $M$, the rotation of $\ker\alpha$ along the $v$-orbit in a transported frame exceeds $\pi$.\footnote{It is in fact then infinite}\\
\end{definition}

\noindent
Our second assumption is therefore:
$$
\begin{array}{ll}
(B) &  \mbox{$\ker\alpha$ turns well along $v$}
\end{array}
$$
In this paper, we will prove the following

\begin{theorem}\label{mainteo}
Let $(M,\alpha)$ be a contact closed manifold. Then under the assumptions (A) and (B), the injection
$$\mathcal{C}_{\beta}\hookrightarrow \Lambda(S^{1},M)$$
is an $S^1$-equivariant homotopy equivalence.
\end{theorem}
\bigskip
\noindent
Let us recall first some properties that we will be using later. Given the contact form $\alpha$, we will let $\xi$ be its Reeb vector field. Namely, $\xi$ is the unique vector satisfying
$$\alpha(\xi)=1, \qquad d\alpha(\xi,\cdot)=0$$
Therefore the following holds (see \cite {B1}):
\begin{lemma}[\cite{B1}]
Under the assumption (A), let $w$ be the Reeb vector field of the 1-form $\beta$, then there exist two functions $\tau$ and $\overline{\mu}$ such that:
$$[\xi,[\xi,v]]=-\tau v, \qquad w=-[\xi,v]+\overline{\mu}\xi$$
where $\overline{\mu}=d\alpha(v,[v,[\xi,v]])$.
\end{lemma}

\noindent
Notice also that with the previous notation, the following holds:
$$\dot{x}=a\xi + bv, \quad \forall x \in \mathcal{L}_{\beta}$$
Moreover if $x$ is in $\mathcal{C}_{\beta}$ then $a$ is a positive constant. One can show (see \cite{B1}) that $\mathcal{C}_{\beta}\setminus M$ has a Hilbert manifold structure. For $x\in \mathcal{C}_{\beta}$, the tangent space at the curve $x$ is given by the set of vector fields
$$Z=\lambda \xi+\mu v +\eta w$$
with the coefficients $\lambda$, $\mu$ and $\eta$ satisfying the following equations:
\begin{equation}
\left \{
\begin{array}{ll}
\dot{\overline{\lambda + \overline{\mu} \eta}}=b \eta - \int_{0}^{1}b \eta, \\
\\
\dot{\eta}=\mu a - \lambda b
\end{array}
\right.
\end{equation}
where $\lambda$, $\mu$ and $\eta$ are 1-periodic.\\
The proof of the main theorem requires several steps. We apply first Smale's theorem to conclude that the injection $\mathcal{L}_{\beta}\hookrightarrow\Lambda(S^{1},M)$ is a homotopy equivalence. Next, we introduce an intermediate space $\mathcal{C}_{\beta}^{+}$ defined by
$$\mathcal{C}_{\beta}^{+}=\left \{x\in \mathcal{L}_{\beta}; \alpha(\dot{x})\geq 0 \right \}, $$
and we show that we can deform $\mathcal{L}_{\beta}$ to $\mathcal{C}_{\beta}^{+}$. This deformation is not continuous because ``Dirac masses'' (see below) along $v$ are created through this transformation. We will ``solve the Dirac masses'', showing how they are created along a smooth deformation in $\mathcal{L}_{\beta}$.\\
In a next and last step we ``push'' the curves of $\mathcal{C}_{\beta}^{+}$ into $\mathcal{C}_{\beta}$. This will be completed by constructing a flow that brings curves with $a\geq 0$ to curves with $a>0.$ \\

\noindent
Before going into more details, let us discuss the assumptions and let us give some examples of contact structures for which they hold.\\
Assumption (A) holds for a number of contact structures with suitable vector fields $v$ in their kernel. For instance the standard contact form on $S^3$
$$\alpha_{0} = x_2dx_1 - x_1dx_2 + x_4dx_3 - x_3dx_4$$
and also the family of contact structures on $T^3$ given by
$$\alpha_{n} = cos(2n\pi z)dx + sin(2n\pi z)dy$$
All the contact forms in the previous examples are tight, but there are also overtwisted contact forms satisfying (A). This is the case of the first non-standard 1-form on $S^3$, given by Gonzalo-Varela in \cite{GonVar}:
$$\alpha_1 =-(\cos(\frac{\pi}{4}+\pi(x_{3}^{2}+x_{4}^{2}))(x_2dx_1-x_1dx_2)+\sin(\frac{\pi}{4}+\pi(x_{3}^{2}+x_{4}^{2}))(x_4dx_3-x_3dx_4))$$
where an (explicit) existence of a suitable $v$ satisfying (A) is proved in \cite{Mart}.\\

\noindent
The assumption (B) holds also for the previous mentioned examples; moreover this assumption has a deeper meaning. In fact, it was proved in the work of Gonzalo \cite{G}, that (B) holds if and only if $\alpha$ extends to a contact circle, namely there exists another contact form $\alpha_{2}$ transverse to $\alpha$ with intersection the line spanned by $v$, such that
$$cos(s)\alpha + sin(s)\alpha_{2}$$
is a contact form for every $s\in \mathbb{R}$. \\
Let us observe that $\alpha_{1}$ defined above represents the first example of an overtwisted contact circle on a compact manifold. In fact, see \cite{G-G} for a question of Geigs and Gonzalo, where they give an example of an overtwisted contact circle on $\mathbb{R}^{3}$ and they observe that they don't know an explicit example of overtwisted contact circle on a compact closed manifold. $\alpha_{1}$ with the $v$ found in \cite{Mart} is such an example.\\
Using this criteria one can give some conditions under which (B) holds:
\begin{lemma}
Assume that (A) holds, then (B) holds if one of the following conditions is satisfied:
$$
\begin{array}{ll}
(i) & |\overline{\mu}|< 2\\
\\
(ii) &  \mbox{there exists a map $u$ on M such that $\overline{\mu} = u_{v}$}
\end{array}
$$
Moreover, if $\overline{\mu}= 0$ then $\alpha$ is tight.
\end{lemma}

\begin{proof}
We use the characterization stated above for contact circles.\\
Let $s$ be a real number, and consider the 1-form
$$\alpha_{s} = \cos(s)\alpha + \sin(s)\beta;$$
then
$$\alpha_{s} \wedge d\alpha_{s} = \cos^{2}(s)\alpha \wedge d\alpha + \sin^{2}(s)\beta\wedge d\beta + \cos(s)\sin(s)(\alpha \wedge d\beta + \beta \wedge d\alpha)$$
Notice now, (see \cite{B1}), that $\alpha \wedge d\beta(\xi, v,w) = -\overline{\mu}$, thus we have
$$\alpha_{s}\wedge d\alpha_{s}(\xi, v,w) = 1 - \frac{\sin(2s)}{2}\overline{\mu}$$
and the conclusion follows for $(i)$.\\
For $(ii)$ we consider
$$\alpha_{s} = \cos(s)\alpha + \sin(s)e^{u}\beta$$
and the same computation yields
$$\alpha_{s}\wedge d\alpha_{s} = \cos^{2}(s)\alpha\wedge d\alpha+e^{2u}\sin^{2}(s)\beta\wedge d\beta+\sin(s)\cos(s)e^{u}(\alpha\wedge d\beta+\alpha\wedge du\wedge \beta)$$
Evaluating at $(\xi, v, [\xi, v])$ we get:
$$\alpha_{s}\wedge d\alpha_{s} = \cos^{2}(s) + e^{2u}\sin^{2}(s) + e^{u}\sin(s)\cos(s)(u_{v}-\overline{\mu})$$
therefore $(ii)$ follows.\\
Now notice that if $\overline{\mu} = 0$ then we have what it is called a taut contact circle (in fact
we have a Cartan structure), therefore based on the result of Geigs-Gonzalo \cite{G-G}, we have that $\alpha$ and $\beta$ are tight.
\end{proof}

\noindent
\textbf{Acknowledgement} This paper is part of the first author's PhD thesis
work under the supervision of A.Bahri, and he would like to thank him for the fruitful discussions and advices that led to the results of this paper.\\
Moreover part of this paper was completed during the year that the second author spent at the Mathematics Department of Rutgers University: the author wishes to express his gratitude for the hospitality and he is grateful to the Nonlinear Analysis Center for its support.

\section{Regularization}
Since we are considering curves in $\mathcal{L}_{\beta}$ that we want to lift to $C_{\beta}$, the first difficulty that we will face are the degenerate curves, namely curves that are not generic in the sense that the components of the tangent vectors can have bad behavior. Therefore, in this section we want to regularize the curves starting from a compact set of $\mathcal{L}_{\beta}$. This regularization will be done by the use of a flow on the curves that induces a heat flow on the components of the tangent vector making them smooth and having isolated zeros.\\
The flow will be constructed on the tangent as a heat flow, but there is no guaranty that the flow is indeed a flow on curves. For that we will first approximate the deformation vector with a smooth one for which we know the local existence. Then we will show that when our approximation tends to the original flow, the maximal time of existence is bounded from below independently of the approximation. These statements will be made precise and clear in what follows.\\

\noindent
Here we want to construct a flow on $\mathcal{L}_{\beta}$ that deforms compact sets of curves $y$ of $\mathcal{L}_{\beta}$ into curves $x$ with $\dot{x}=a\xi+bv$, where $a$ and $b$ are smooth and have zeros of finite order. This will be used in our proof below. In fact what is needed is just the fact that $a$ has zeros of finite order but here we will prove the stronger result involving the preservation of the number of zeros of $b$ along the deformation.\\
A first idea is to consider the flow defined by the vector field
$$Z=(\dot{a}+f)\xi + (\dot{b}+g )v +\eta [\xi,v]$$
where $\eta=\eta(a,b)$ satisfies the following differential equation: $$\dot{\eta}=\overline{\mu}b\eta+\dot{b}a-\dot{a}b+ga-fb.$$
The vector field $Z$ constructed in this way will generate a diffusion flow on $a$ and $b$ as follow:
\begin{equation}\label{lim}
\left\{ \begin{array}{ll}
\displaystyle\frac{\partial a}{\partial s}=\ddot{a}+\dot{f}-b\eta \\
\\
\displaystyle\frac{\partial b}{\partial s}=\ddot{b}+\dot{g}+\eta(a\tau-\overline{\mu}_{\xi} b)
\end{array}
\right.
\end{equation}
To ensure the periodicity of $\eta$, we set $f=-\kappa b$ and $g=\kappa a$ with $\kappa=\kappa(a,b)$ satisfying $$\int_{0}^{1}e^{-\int_{0}^{r}b(u,s)\overline{\mu}(u)du}(a\dot{b}-\dot{a}b)(r,s)dr+ \kappa\int_{0}^{1}e^{-\int_{0}^{r}b(u,s)\overline{\mu}(u)du}(a^{2}+b^{2})(r,s)dr=0$$
The problem with this first attempt is that the previous system depends on the curve and we do not know so far how this vector $Z$ acts on the curve and if it defines indeed a flow on $\mathcal{L}_{\beta}$. We will follow the same technique as in \cite{B4}, to prove that indeed we have a flow on the curves that gives rise to the system defined above. Hence a first step consists of regularizing $a$ and $b$ by using a mollifier $\phi_{\varepsilon}$ and we use the classical Cauchy-Lipschitz theorem for the flow defined by the approximated vector field $Z_{\varepsilon}$. The second part consists of showing the convergence to the aimed system as $\varepsilon$ converges to zero.

\subsection{Approximated Flow}

\noindent
We consider the regularizing operator $\phi_{\varepsilon}:\mathcal{H}^{1}(S^{1}) \longrightarrow \mathcal{H}^{2}(S^{1})$ such that for $f\in \mathcal{H}^{1}(S^{1})$, $\phi_{\varepsilon}(f)$ satisfies the following equation:
\begin{equation}\label{trans}
-\varepsilon\ddot{\phi}_{\varepsilon}(f)+\phi_{\varepsilon}(f)=f.
\end{equation}
Notice that in terms of Fourier coefficients this corresponds to divide by $1+\varepsilon k^{2}$.

\subsubsection{Preliminary estimates}

\noindent
We have the following estimates for the operator $\phi_{\varepsilon}$.
\begin{lemma}
Let $f\in \mathcal{H}^{1}(S^{1})$, then there exists $C>0$ such that\\

(i) $\|\phi_{\varepsilon}(f)\|_{\mathcal{H}^{1}}\leq C \|f\|_{\mathcal{H}^{1}}$\\

(ii) $\|\dot{\phi}_{\varepsilon}(f)\|_{\mathcal{H}^{1}}^{2}\leq \frac{C}{\varepsilon}\|f\|_{\mathcal{H}^{1}}^{2}$\\

(iii) $\|\phi_{\varepsilon}(f)-f\|_{L^{2}}^{2}\leq C\varepsilon \|f\|_{\mathcal{H}^{1}}^{2}$\\

\end{lemma}

\begin{proof}
(i) As it was defined $\phi_{\varepsilon}(f)$ satisfies
$$-\varepsilon \ddot{\phi}_{\varepsilon}(f)+\phi_{\varepsilon}(f)=f$$
So if we multiply the previous equation by $\phi_{\varepsilon}(f)$ we have
$$\|\phi_{\varepsilon}(f)\|_{L^{2}}^{2}\leq \|f\|_{L^{2}}\|\phi_{\varepsilon}(f)\|_{L^{2}}-\varepsilon \|\dot{\phi}_{\varepsilon}(f)\|_{L^{2}}^{2}$$
hence the inequality follows for the $L^{2}$ norm, and by linearity we have the same inequality for $\dot{f}$.\\
(ii) With the same idea, we find
$$\varepsilon \|\dot{\phi}_{\varepsilon}(f)\|_{L^{2}}^{2}\leq C \|f\|_{L^{2}}^{2}$$
Therefore the estimate follows by linearity.\\
(iii) From (\ref{trans}) we have that
$$\int_{0}^{1}|f-\phi_{\varepsilon}(f)|^{2}=\varepsilon \|\dot{\phi}_{\varepsilon}(f)\|_{L^{2}}^{2} -\varepsilon \|\dot{f}\|_{L^{2}}\|\dot{\phi}_{\varepsilon}(f)\|_{L^{2}}$$
Using (i) we have
$$\|f-\phi_{\varepsilon}(f)\|_{L^{2}}^{2}\leq 2 \varepsilon \|\dot{f}\|_{L^{2}}^{2}.$$
\end{proof}

\noindent
We consider now the operator $L_{\varepsilon}$ defined by
$$L_{\varepsilon}(f)(s,t)=\sum e^{-s\frac{k^{2}}{\varepsilon k^{2}+1}}e^{ikt}f_{k},$$
where $f=\sum f_{k}e^{ikt}$. This operator satisfies for $g=L_{\varepsilon}(f)$,
$$\left\{\begin{array}{ll}
\displaystyle\frac{\partial g}{\partial s}-\ddot{\phi_{\varepsilon}}(g)=0\\
\\
g(0,t)=f(t)
\end{array}
\right. $$
Notice that $L_{0}$ corresponds to the inverse of the homogeneous heat operator.

\begin{lemma}
The operator $L_{\varepsilon}$ converges to $L_{0}$ in the operator norm from $\mathcal{H}^{1}(S^{1})$ to $L^{\infty}(0,\infty,\mathcal{H}^{1}(S^{1}))$.
\end{lemma}

\begin{proof}
Let $f \in \mathcal{H}^{1}(S^{1})$, then $$\|(L_{\varepsilon}-L_{0})f\|_{\mathcal{H}^{1}}^{2}=\sum (1+k^{2})|f_{k}|^{2}|e^{-s\frac{k^{2}}{\varepsilon k^{2}+1}}-e^{-sk^{2}}|^{2}$$
$$\leq C\varepsilon^{2}\|f\|_{\mathcal{H}^{1}}^{2}$$
where $C$ is independent of $s$ and $\varepsilon$. Hence $\|L_{\varepsilon}-L_{0}\|_{L^{\infty}(0,\infty,\mathcal{H}^{1})}$ converges to zero as $\varepsilon$ tends to zero with a rate of at least $\varepsilon$.
\end{proof}

\noindent
A similar lemma holds for the operator $\tilde{L}_{\varepsilon}$ corresponding to the general solution of: $$\left\{\begin{array}{ll}
\displaystyle\frac{\partial g}{\partial s}-\ddot{\phi_{\varepsilon}}(g)=f\\
\\
g(0,t)=0
\end{array}
\right. $$

\begin{lemma}
Let $f\in L^{\infty}(0,\epsilon,\mathcal{H}^{l}(S^{1}))$, then $g=\tilde{L}_{\varepsilon}(f)\in L^{\infty}(0,\epsilon,\mathcal{H}^{l}(S^{1})$. Moreover, there exists $C$ independent of $\varepsilon$, such that
$$\|\tilde{L}_{\varepsilon}(f)\|_{\mathcal{H}^{l}}^{2}(s)\leq C \int_{0}^{s} \big(\|f\|_{\mathcal{H}^{l-1}}^{2}(r)+
\varepsilon \|f\|_{\mathcal{H}^{l}}^{2}(r)\big) dr$$
for all $l\geq 1$.
\end{lemma}

\begin{proof}
We consider the Fourier expansion of $f=\sum_{k}f_{k}(s)e^{ikt}$, then for $u=\tilde{L}_{\varepsilon}(f)$ we have that $$u_{k}(s)=\int_{0}^{s}e^{-\frac{k^{2}}{1+\varepsilon k^{2}}(s-r)}f_{k}(r)dr$$
Thus
$$|u_{k}^{2}|\leq  \frac{1+\varepsilon k^{2}}{k^{2}}\int_{0}^{s}|f_{k}|^{2}(r)dr$$
Therefore
$$ \|u\|_{\mathcal{H}^{l}}^{2}\leq C\int_{0}^{s}\|f\|_{\mathcal{H}^{l-1}}^{2}(r)+\varepsilon \|f\|_{\mathcal{H}^{l}}^{2}(r)dr$$
\end{proof}

\subsubsection{Estimates along the flow}

\noindent
We consider now the component
$$ \eta(x,a,b)=e^{\int_{0}^{t}b(u,s)\overline{\mu}(u)du}\Big(\int_{0}^{t}e^{-\int_{0}^{r}b(u,s)\overline{\mu}(u)du}
((a\dot{b}-\dot{a}b)+\kappa (a^{2}+b^{2})(r,s)dr\Big)$$
where $$\kappa(x,a,b)=-\frac{\int_{0}^{1}e^{-\int_{0}^{r}b(u,s)\overline{\mu}(u)du}
(a\dot{b}-\dot{a}b)(r,s)dr}{\int_{0}^{1}e^{-\int_{0}^{r}b(u,s)\overline{\mu}(u)du}(a^{2}+b^{2})(r,s)dr}.$$
The constant $\kappa$ is computed so that $\eta$ is 1-periodic. Notice that we use $(x,a,b)$ instead of $(x,\dot{x})$ since we are interested more in the coefficients $a$ and $b$.\\
Similarly, we take $\lambda(x,a,b)=\dot{a}-\kappa b$ and $\mu(x,a,b)=\dot{b}+\kappa a$.\\
We define now
$$\eta_{\varepsilon}(x,a,b)=\eta(x,\phi_{\varepsilon}(a),\phi_{\varepsilon}(b)), \qquad \lambda_{\epsilon}(x,a,b)=\lambda(x,\phi_{\varepsilon}(a),\phi_{\varepsilon}(b))$$
$$\mu_{\epsilon}(x,a,b)=\mu(x,\phi_{\varepsilon}(a),\phi_{\varepsilon}(b)), \qquad
\kappa_{\varepsilon}(x,a,b)=\kappa(x,\phi_{\varepsilon}(a),\phi_{\varepsilon}(b))$$
The vector field
$$Z_{\varepsilon}=\lambda_{\varepsilon}\xi + \mu_{\varepsilon}v+\eta_{\varepsilon}[\xi,v]$$
is then locally Lipschitz and hence the flow
\begin{equation}\label{flow}
\left\{\begin{array}{ll}
\displaystyle\frac{\partial x}{\partial s}= Z_{\varepsilon}(x)=\lambda_{\varepsilon}\xi+\mu_{\varepsilon}v+\eta_{\varepsilon}[\xi,v]\\
\\
x(0)=x_{0}
\end{array}
\right.
\end{equation}
has a unique solution that exists in $[0,s_{0}(\varepsilon))$, by the standard Cauchy-Lipschitz theorem. It is important to notice that this flow will not stay in $\mathcal{L}_{\beta}$, in fact it will be defined in a neighborhood of $x_{0}$ in $\mathcal{H}^{2}(S^{1},M)$. But hopefully, when $\varepsilon$ converges to zero the limiting flow will be in $\mathcal{L}_{\beta}$.\\
We want to have good estimates on the coefficients $a$ and $b$ as $\varepsilon$ converges to zero.\\
Using these notations we have that under the flow generated by $Z_{\varepsilon}$ (see \cite{B4}):
\begin{equation}\label{syst}
\left\{ \begin{array}{llll}
\displaystyle\frac{\partial a}{\partial s}=\dot{\lambda}_{\varepsilon}-b\eta_{\varepsilon}+c\mu_{\varepsilon}\\
\\
\displaystyle\frac{\partial b}{\partial s}=\dot{\mu}_{\varepsilon}+(\tau a-\overline{\mu}_{\xi}b)\eta_{\varepsilon}+c(\tau \lambda_{\varepsilon}-\overline{\mu}_{\xi}\mu_{\varepsilon})\\
\\
\displaystyle\frac{\partial c}{\partial s}=\dot{\eta}_{\varepsilon}-\overline{\mu} b \eta_{\varepsilon}-\mu_{\varepsilon}a+\lambda_{\varepsilon}b-\overline{\mu}\mu_{\varepsilon} c\\
\\
a(0)=a_{0}, \; b(0)=b_{0}, \; c(0)=0.
\end{array}
\right.
\end{equation}
where $c$ is the component of $\dot{x}$ along $[\xi,v]$. Existence is not found directly from the system itself, but instead it follows from the one in (\ref{flow}), since this is the evolution of the components of the tangent vector to the curve evolving under the flow generated by $Z_{\varepsilon}$. These functions $\overline{\mu}$, $\tau$, are given as functions of $s$. That is $\overline{\mu}(x(s))$ and $\tau(x(s))$, where $x(s)$ is the solution of (\ref{flow}). We then reformulate the previous system as a fixed point problem. We will then derive appropriate estimates on (\ref{syst}) that will allow us to establish existence of the limiting flow as $\varepsilon\rightarrow0$.\\
For this purpose, we consider the operator $F_{\varepsilon}$ defined by : $$F_{\varepsilon}(x,\dot{x})=\left[\begin{array}{ccc}
-\kappa_{\varepsilon}\dot{\phi}_{\varepsilon}(b)-b\eta_{\epsilon}+c\mu_{\varepsilon}\\
\kappa_{\varepsilon}\dot{\phi}_{\varepsilon}(a)+\eta_{\varepsilon}(a\tau-\overline{\mu}_{\xi} b)+c(\tau \lambda_{\varepsilon}-\overline{\mu}_{\xi}\mu_{\varepsilon})\\
\dot{\eta}_{\varepsilon}-\overline{\mu} b \eta_{\varepsilon}-\mu_{\varepsilon}a+\lambda_{\varepsilon}b -\overline{\mu}\mu_{\varepsilon} c
\end{array}
\right]
$$
This evolution equation follows from the proposition in Appendix. \\
Now we define the space $\mathcal{B}_{\varepsilon}$, for $\varepsilon>0$, as follows:
$$\mathcal{B}_{\varepsilon}=\{ x(s,t)\in L^{\infty}(0,\varepsilon,\mathcal{H}^{2}(S^{1}));\dot{x}(s,t)\in L^{\infty}(0,\varepsilon,\mathcal{H}^{1}(S^{1}))\}.$$
So $F_{\varepsilon}$ sends $\mathcal{B}_{\epsilon}$ to itself. Also, we can define the operator $T_{\varepsilon}$ by $$T_{\varepsilon}(x,\dot{x})=L_{\varepsilon}\left[\begin{array}{ccc}
a_{0}\\
b_{0}\\
0
\end{array}
\right]+
\left[\begin{array}{ccc}
\tilde{L}_{\varepsilon}(-\kappa_{\varepsilon}\dot{\phi}_{\varepsilon}(b)-b\eta_{\epsilon}+c_{\varepsilon}\mu_{\varepsilon})\\
\tilde{L}_{\varepsilon}(\kappa_{\varepsilon}\dot{\phi}_{\varepsilon}(a)+\eta_{\varepsilon}(a\tau-\overline{\mu}_{\xi} b)+c_{\varepsilon}(\tau\lambda_{\varepsilon}-\overline{\mu}_{\xi}\mu_{\varepsilon}))\\
\int_{0}^{s}e^{\int_{s}^{r}(\overline{\mu}\mu_{\varepsilon})dr}(\dot{\eta}_{\varepsilon}-\overline{\mu} b \eta_{\varepsilon}-\mu_{\varepsilon}a+\lambda_{\varepsilon}b) ds
\end{array}
\right]
$$
So that the fixed point of $T_{\varepsilon}$ corresponds to the solution of the system (\ref{syst}). In fact we have
$$\frac{\partial}{\partial t}T_{\varepsilon}(x,\dot{x})=\ddot{\phi}_{\varepsilon}(L_{\varepsilon}\left[\begin{array}{ccc}
a_{0}\\
b_{0}\\
0
\end{array}
\right])+
\left[\begin{array}{ccc}
\ddot{\phi}_{\varepsilon}(\tilde{L}_{\varepsilon}
(-\kappa_{\varepsilon}\dot{\phi}_{\varepsilon}(b)-b\eta_{\varepsilon}+c_{\varepsilon}\mu_{\varepsilon}))\\
\ddot{\phi}_{\varepsilon}(\tilde{L}_{\varepsilon}(\kappa_{\varepsilon}\dot{\phi_{\varepsilon}}(a)+
\eta_{\varepsilon}(a\tau-\overline{\mu}_{\xi} b)+c_{\varepsilon}(\tau\lambda_{\varepsilon}\overline{\mu}_{\xi}\mu_{\varepsilon})))\\
0
\end{array}
\right]+
$$
$$+\left[\begin{array}{ccc}
-\kappa_{\varepsilon}\dot{\phi}_{\varepsilon}(b)-b\eta_{\epsilon}+c_{\varepsilon}\mu_{\varepsilon}\\
\kappa_{\varepsilon}\dot{\phi}_{\varepsilon}(a)+\eta_{\varepsilon}(a\tau-\overline{\mu}_{\xi} b)+c_{\varepsilon}(\tau\lambda_{\varepsilon}-\overline{\mu}_{\xi}\mu_{\varepsilon})\\
\dot{\eta}_{\varepsilon}-\overline{\mu} b \eta_{\varepsilon}-\mu_{\varepsilon}a+\lambda_{\varepsilon}b-\overline{\mu}\mu_{\varepsilon} c
\end{array}
\right].$$

\noindent
In what follow we will use $\|f\|$ instead of $\|f(x(\cdot))\|$ for the functions depending on the curve (such as $\tau$, $\overline{\mu}$, etc.).

\begin{lemma}
Let $x\in \mathcal{B}_{\epsilon}$ such that $\|\dot{x}\|_{L^{2}}\geq \delta >0$, then there exist three positive constants $C_{1}$, $C_{2}$ and $C_{3}$ independent of $\varepsilon$ and possibly depending on $\delta$, such that
$$\|T_{\varepsilon}(x,\dot{x})\|_{\mathcal{H}^{1}}(s)\leq \|\dot{x}_{0}\|_{\mathcal{H}^{1}}+ e^{C_{3}\int_{0}^{s}\|\dot{x}\|_{\mathcal{H}^{1}}}\int_{0}^{s}(C_{1}+C_{2}\sqrt{\varepsilon})e^{C_{3}\|\dot{x}\|_{\mathcal{H}^{1}}(r)}dr$$
\end{lemma}

\begin{proof}
We first need an estimate on $\lambda_{\varepsilon}$, $\mu_{\varepsilon}$ and $\eta_{\varepsilon}$. In order to do that, an estimate on the variable $\kappa_{\varepsilon}$ is necessary.\\
By the very definition of $\kappa$ we get
$$|\kappa_{\varepsilon}|\leq C \frac{(\|a\|_{\mathcal{H}^{1}}+\|b\|_{\mathcal{H}^{1}})^{2}e^{2\|b\|_{\mathcal{H}^{1}}\|\overline{\mu}\|_{L^{\infty}}}}{\|a\|_{L^{2}}^{2}+\|b\|_{L^{2}}^{2}}$$
Using the previous estimate and Lemma (2.1.), we have
$$\|\lambda_{\varepsilon}(x,\dot{x})\|_{\mathcal{H}^{1}}\leq \frac{1}{\sqrt{\varepsilon}}\|a\|_{\mathcal{H}^{1}}+
\frac{(\|a\|_{\mathcal{H}^{1}}+\|b\|_{\mathcal{H}^{1}})^{2}e^{2\|b\|_{\mathcal{H}^{1}}\|\overline{\mu}\|_{L^{\infty}}}}
{\|a\|_{L^{2}}^{2}+\|b\|_{L^{2}}^{2}}\|b\|_{\mathcal{H}^{1}}$$
A similar estimate holds for $\mu_{\varepsilon}$, that is
$$\|\mu_{\varepsilon}(x,\dot{x})\|_{\mathcal{H}^{1}}\leq \frac{1}{\sqrt{\varepsilon}}\|b\|_{\mathcal{H}^{1}}+
\frac{(\|a\|_{\mathcal{H}^{1}}+\|b\|_{\mathcal{H}^{1}})^{2}e^{2\|b\|_{\mathcal{H}^{1}}\|\overline{\mu}\|_{L^{\infty}}}}
{\|a\|_{L^{2}}^{2}+\|b\|_{L^{2}}^{2}}\|a\|_{\mathcal{H}^{1}}$$
The estimate for $\eta_{\varepsilon}$ is a little bit different
$$\|\eta_{\varepsilon}\|_{\mathcal{H}^{2}}\leq C(\|b\|_{\mathcal{H}^{1}}(\|b\|_{\mathcal{H}^{1}}+\|a\|_{\mathcal{H}^{1}}+\|c\|_{\mathcal{H}^{1}}+1)e^{2\|b\|_{\mathcal{H}^{1}}\|\overline{\mu}\|_{L^{\infty}}}
\Big(\frac{1}{\varepsilon}(\|a\|_{\mathcal{H}_{1}}+\|b\|_{\mathcal{H}^{1}})^{2}+
\frac{\|a\|_{\mathcal{H}^{1}}+\|b\|_{\mathcal{H}^{1}}}{\|a\|_{L^{2}}+\|b\|_{L^{2}}}\Big)$$
And by taking again another derivative we have the desired estimate. In fact we have
$$\|\dot{\eta}_{\varepsilon}\|_{L^{2}} \leq Ce^{\|b\|_{\mathcal{H}^{1}}\|\overline{\mu}\|_{L^{\infty}}}
(\|\overline{\mu}\|_{L^{\infty}}\|b\|_{L^{2}}+e^{\|b\|_{\mathcal{H}^{1}}\|\overline{\mu}\|_{L^{\infty}}}
(\|a\|_{\mathcal{H}^{1}}^{2}+\|b\|_{\mathcal{H}^{1}}^{2}+|\kappa_{\varepsilon}|(\|a\|_{L^{2}}^{2}+\|b\|_{L^{2}}^{2}))$$
Now
$$\|-\kappa_{\varepsilon}\dot{\phi}_{\varepsilon}(b)- b\eta_{\epsilon}+c\mu_{\varepsilon}\|_{\mathcal{H}^{1}}\leq C\big(\frac{1}{\sqrt{\varepsilon}}|\kappa_{\varepsilon}\||b\|_{\mathcal{H}^{1}}+\|b\|_{\mathcal{H}^{1}}\|\eta_{\varepsilon}\|_{\mathcal{H}^{1}}+
\|c\|_{\mathcal{H}^{1}}\|\mu_{\varepsilon}\|_{\mathcal{H}^{1}}\big)$$
Also
$$\|\kappa_{\varepsilon}\dot{\phi}_{\varepsilon}(a)+\eta(a\tau-\overline{\mu}_{\xi}b)+
c_{\varepsilon}(\tau\lambda_{\varepsilon}-\overline{\mu}_{\xi}\mu_{\varepsilon})\|_{\mathcal{H}^{1}}\leq C\big(\frac{1}{\sqrt{\varepsilon}}|\kappa_{\varepsilon}|\|a\|_{\mathcal{H}^{1}}+$$
$$+\|\dot{x}\|_{L^{2}}\|\eta_{\varepsilon}\|_{\mathcal{H}^{1}}(\|a\|_{\mathcal{H}^{1}}+\|b\|_{\mathcal{H}^{1}})+
\|\dot{x}\|_{L^{2}}\|c\|_{\mathcal{H}^{1}}(\|\mu_{\varepsilon}\|_{\mathcal{H}^{1}}+\|\lambda_{\varepsilon}\|_{\mathcal{H}^{1}})\big)$$
It is crucial to notice that $\eta_{\varepsilon}$ satisfies $$\dot{\eta}_{\varepsilon}=\overline{\mu}\eta_{\varepsilon}+\phi_{\varepsilon}(a)\mu_{\varepsilon}-\lambda_{\varepsilon}\phi_{\varepsilon}(b)$$
Therefore we have
$$\|\dot{\eta}_{\varepsilon}-\overline{\mu} b\eta_{\varepsilon}-\mu_{\varepsilon}a+\lambda_{\varepsilon}b\|_{\mathcal{H}^{1}}=
\|\overline{\mu}\eta_{\varepsilon}(\phi_{\varepsilon}(b)-b)+\mu_{\varepsilon}(\phi_{\varepsilon}(a)-a)+
\lambda_{\varepsilon}(b-\phi_{\varepsilon}(b))\|_{\mathcal{H}^{1}}$$
$$\leq \|\overline{\mu}\|_{\mathcal{H}^{1}}\|\eta_{\varepsilon}\|_{\mathcal{H}_{1}}\|\phi_{\varepsilon}(b)-b\|_{L^{2}}+ \|\overline{\mu}\|_{L^{2}}\|\eta_{\varepsilon}\|_{L^{2}}\|b\|_{\mathcal{H}^{1}}+$$
$$+\|\mu_{\varepsilon}\|_{L^{2}}\|a\|_{\mathcal{H}^{1}}+\|\mu_{\varepsilon}\|_{\mathcal{H}^{1}}\|a-\phi_{\varepsilon}(a)\|_{L^{2}}+ $$
$$+\|\lambda_{\varepsilon}\|_{\mathcal{H}^{1}}\|b-\phi_{\varepsilon}(b)\|_{L^{2}}+\|\lambda_{\varepsilon}\|_{L^{2}}\|b\|_{\mathcal{H}^{1}}$$
Using Lemma (2.1.) we have
$$\|\dot{\eta}_{\varepsilon}-\overline{\mu} b \eta_{\varepsilon}-\mu_{\varepsilon}a+\lambda_{\varepsilon}b\|_{\mathcal{H}^{1}}\leq C\big(\sqrt{\varepsilon}(\|\overline{\mu}\|_{\mathcal{H}^{1}}\|\eta_{\varepsilon}\|_{\mathcal{H}_{1}}\|b\|_{\mathcal{H}^{1}}+
\|\mu_{\varepsilon}\|_{\mathcal{H}^{1}}\|a\|_{\mathcal{H}^{1}}+\|\lambda_{\varepsilon}\|_{\mathcal{H}^{1}}\|b\|_{\mathcal{H}^{1}})+$$
$$+(\|\overline{\mu}\|_{L^{\infty}}\|\eta_{\varepsilon}\|_{L^{2}}\|b\|_{\mathcal{H}^{1}}+
\|\mu_{\varepsilon}\|_{L^{2}}\|a\|_{\mathcal{H}^{1}}+\|\lambda_{\varepsilon}\|_{L^{2}}\|b\|_{\mathcal{H}^{1}})\big)$$
Now for the $L^{2}$ norm we have
$$\|\dot{\eta}_{\varepsilon}-\overline{\mu} b \eta_{\varepsilon}-\mu_{\varepsilon}a+\lambda_{\varepsilon}b\|_{L^{2}}\leq C\big(\sqrt{\varepsilon}(\|\overline{\mu}\|_{\mathcal{H}^{1}}\|\eta_{\varepsilon}\|_{\mathcal{H}_{1}}\|b\|_{\mathcal{H}^{1}}+
\|\mu_{\varepsilon}\|_{\mathcal{H}^{1}}\|a\|_{\mathcal{H}^{1}}+\|\lambda_{\varepsilon}\|_{\mathcal{H}^{1}}\|b\|_{\mathcal{H}^{1}})\big)$$
Let us set
$$ Y=\dot{\eta}_{\varepsilon}-\overline{\mu} b \eta_{\varepsilon}-\mu_{\varepsilon}a+\lambda_{\varepsilon}b, \qquad A=-\overline{\mu}\mu_{\varepsilon}$$
So if we write a curve $x\in \mathcal{H}^{2}(S^{1},M)$ as
$$\dot{x}=a\xi +bv +c [\xi,v]$$
then $c$ satisfies
$$\frac{\partial}{\partial s}c=Ac+Y$$
From this equality we have
$$\frac{\partial}{\partial s}\|c\|_{L^{2}}\leq C \|A\|_{L^{2}}\|c\|_{L^{2}}+\|Y\|_{L^{2}}$$
thus
$$\|c\|_{L^{2}}\leq C \int_{0}^{s}e^{\int_{r}^{s}\|A\|_{L^{2}}dr}\|Y\|_{L^{2}}$$
$$\leq C\sqrt{\varepsilon}e^{\int_{0}^{s}\|\dot{x}\|_{\mathcal{H}^{1}}}\int_{0}^{s}\|\dot{x}\|_{\mathcal{H}^{1}}^{3}ds$$
In a similar manner we have
$$\frac{\partial }{\partial s}(\|\dot{c}\|_{L^{2}})\leq
C \|A\|_{L^{2}}\|\dot{c}\|_{L^{2}}+\|A\|_{\mathcal{H}^{1}}\|c\|_{L^{2}}+\|Y\|_{\mathcal{H}^{1}}$$
Thus
$$\|\dot{c}\|_{L^{2}}\leq C e^{C\int_{0}^{s}\|\dot{x}\|_{\mathcal{H}^{1}}}\int_{0}^{s}\|\dot{x}\|_{\mathcal{H}^{1}}^{4}ds$$
Hence
$$\|c\|_{\mathcal{H}^{1}}\leq C_{1} e^{C\int_{0}^{s}\|\dot{x}\|_{\mathcal{H}^{1}}} \int_{0}^{s}(1+\|\dot{x}\|_{\mathcal{H}^{1}}^{4})ds$$
And the conclusion of the lemma follows from the estimate that we got on the operator $\tilde{L}_{\varepsilon}$ in Lemma 2.3.
\end{proof}

\noindent
We set $s_{0}(\varepsilon)$ the existence time of the solution of system (\ref{syst}). Then we have the following
\begin{theorem}
There exists $\varepsilon_{0}>0$, and $\sigma>0$ such that for every $\varepsilon<\varepsilon_{0}$, $s_{0}(\varepsilon)>\sigma$.
\end{theorem}

\noindent
The proof follows from a Gronwall type inequality. Let us first, state and prove the general inequality:

\begin{lemma}
Let $y_{\varepsilon}$ be a family of non-negative $C^{1}$ functions such that
$$y_{\varepsilon}\leq C_{\varepsilon} f\Big(\int_{0}^{s}y_\varepsilon(r)dr\Big)$$
for an increasing and positive function $f$. Then the blow-up time of $y_{\varepsilon}$ depends only on $C_{\varepsilon}$. More precisely, if $C_{\varepsilon}$ is bounded then the blow-up time is bounded away from zero.
\end{lemma}
\begin{proof}
For the sake of simplicity we will remove the index $\varepsilon$ for now. Let us set $U=\int_{0}^{s}y$, then one have
$$U'\leq C f(U)$$
Hence if we set $G$ to be the anti-derivative of $\frac{1}{f}$, then we get $$G(U(s))-G(U(0))\leq Cs.$$
Therefore, $$U(s)\leq G^{-1}(Cs+G(0)),$$
thus $$y\leq C f(G^{-1}(Cs+G(0))).$$
And result of the lemma follows.
\end{proof}

\begin{proof}(of Theorem)
We have that the solution $x_{\varepsilon}$ (or  more precisely $\dot{x}_{\varepsilon}$) is a fixed point of $T_{\varepsilon}$. From the previous lemma we have
$$\|\dot{x}_{\varepsilon}\|_{\mathcal{H}^{1}}(s)\leq \|\dot{x}_{0}\|_{\mathcal{H}^{1}}+e^{C_{3}\int_{0}^{s}\|\dot{x}\|_{\mathcal{H}^{1}}}
(C_{1}+\sqrt{\varepsilon}C_{2})\int_{0}^{s}e^{C_{3}\|\dot{x}\|_{\mathcal{H}^{1}}}(r)dr$$
Hence if we set if we set $y=\|\dot{x}\|_{\mathcal{H}^{1}}$, then we have
$$y \leq \|\dot{x}_{0}\|_{\mathcal{H}^{1}}+e^{C_{3}\int_{0}^{s}y}(C_{1}+\sqrt{\varepsilon}C_{2})\int_{0}^{s}e^{C_{3}y}$$
Now by using the Jensen inequality for the convex function $u \longmapsto e^{C_{3}u}$ and assuming that $s\leq 1$ we have
$$y\leq \|\dot{x}_{0}\|_{\mathcal{H}^{1}}+(C_{1}+\sqrt{\varepsilon}C_{2})\big(\int_{0}^{s}e^{C_{3}y}\big)^{2}$$
Setting $H=e^{C_{3}y}$ we get
$$H \leq e^{\|\dot{x}_{0}\|_{\mathcal{H}^{1}}+(C_{1}+\sqrt{\varepsilon} C_{2})(\int_{0}^{s} H 	z)^{2}}$$
so we can apply the previous lemma for $H$ and $f(t)=e^{||\dot{x}_{0}||_{\mathcal{H}^{1}}+(C_{1}+\sqrt{\varepsilon}C_{2})t^{2}}$. Since the constants never blow-up in our case, we have the result of the theorem, that is the blow-up time for $H$ and hence for $\|\dot{x}\|_{\mathcal{H}^{1}}$ is bounded from below independently on $\varepsilon$.
\end{proof}

\subsection{Convergence}

\noindent
Now to see the convergence of the solution as $\varepsilon$ tends to zero, we need the following
\begin{lemma}
If $x_{\varepsilon}$ is the solution of the flow defined above, then: \\

(i) $c_{\varepsilon}$ converges to zero in $L^{2}$\\

(ii) $a_{\varepsilon}$ and $b_{\varepsilon}$ converge strongly in $L^{2}$ to a solution of the flow (\ref{lim}).\\

\end{lemma}

\begin{proof}
(i) Notice that $$c_{\varepsilon}(s)=\int_{0}^{s}e^{\int_{s}^{r}(\overline{\mu}\mu_{\varepsilon})dr}(\overline{\mu}\eta_{\varepsilon}(\phi_{\varepsilon}(b_{\varepsilon})-b_{\varepsilon})+\mu_{\varepsilon}(\phi_{\varepsilon}(a_{\varepsilon})-a_{\varepsilon})+\lambda_{\varepsilon}(b_{\varepsilon}-\phi_{\varepsilon}(b_{\varepsilon}))(r)dr.$$
Since we have boundedness in the $\mathcal{H}^{1}$ sense for $s\in[0,s_{0}]$ we can extract a convergent subsequence in all the $L^{p}$; this, combined with the boundedness of $\eta_{\varepsilon}$, $\mu_{\varepsilon}$ and $\lambda_{\varepsilon}$, gives us the convergence to zero in $L^{2}$.\\
(ii) Let us go back to the fixed point formulation,
$$\dot{x}_{\varepsilon}=T_{\varepsilon}(\dot{x}_{\varepsilon})$$
that is $$a_{\varepsilon}=L_{\varepsilon}(a_{0})+\tilde{L}_{\varepsilon}(-\kappa_{\varepsilon}\dot{\phi}_{\varepsilon}(b_{\varepsilon})- b_{\varepsilon}\eta_{\epsilon}+c_{\varepsilon}\mu_{\varepsilon})$$
The term $  c_{\varepsilon}\mu_{\varepsilon}$ converges strongly to zero in $L^{2}$. Also we have convergence of the term $b_{\varepsilon}\eta_{\epsilon}$ to $b\eta$ and $\kappa_{\varepsilon}$ to $\kappa$. This tells us that $\tilde{L}_{\varepsilon}(\dot{\phi}_{\varepsilon}(b_{\varepsilon}))$ converges in $L^{2}$. So one can check that the limit is $\tilde{L}_{0}(\dot{a})$. The same holds for $b_{\varepsilon}$, hence we can send $\varepsilon$ to $0$ to get a limiting system of the form:

$$\left\{ \begin{array}{ll}
\displaystyle\frac{\partial a}{\partial s}=\ddot{a}-\kappa \dot{b}-b\eta \\
\\
\displaystyle\frac{\partial b}{\partial s}=\ddot{b}+\kappa \dot{a}+\eta(a\tau-\overline{\mu}_{\xi} b)\\
\end{array}
\right.$$
\end{proof}

\noindent
Therefore we proved that the system induces indeed a flow on the curves of $\mathcal{L}_{\beta}$ and it has a regularizing effect on $a$ and $b$ caused by the diffusion operator and this leads to the main result of this section.

\section{From $\mathcal{L}_{\beta}$ to $\mathcal{C}_{\beta}^{+}$}
In this section we will lift curves having isolated zeros of $a$ from a compact set of curves in $\mathcal{L}_{\beta}$ to the space $\mathcal{C}_{\beta}^{+}$. This is where the assumption $(B)$ is crucial. Since flowing along $v$ will always transport vectors with negative $\xi$ components to vectors with positive component. But first we want to give a rigorous definition on the transport along a curve in $\mathcal{L}_{\beta}$ by extending its tangent vector in a small neighborhood. The lifting process is not continuous, in fact, we will show that this procedure will lead to the formation of some ``singularities'' that will be removed later.\\

\subsection{Extending the tangent vector $\dot{x}$ of a curve $x$ in $\mathcal{L}_{\beta}$}

\noindent
First we describe here a way of extending the tangent vector $\dot{x}$ to a curve $x$ in $\mathcal{L}_{\beta}$ which we take with the $\mathcal{H}^{2}$-topology, to a neighborhood of it, near a point $x(t)$ where $\dot{x}(t)$ is non zero. This becomes a vector field in that neighborhood allowing us to define a transport map along the curve. We consider a curve $x\in \mathcal{L}_{\beta}$. Let $S$ a disc at $x(0)$ (we are assuming that $\dot{x}(0)$ is non zero) transverse to the curve and tangent to $[\xi,v]$ at $x(0)$(see figure \ref{fig1} below). Now for any $y_{0}\in S$, we consider the solution of the dynamical system
$$y'(t)=a(t)\xi(y(t))+b(t)v(y(t))$$
with $y(0)=y_{0}$. Since $a$ and $b$ are in $\mathcal{H}^{1}$, by the continuous dependence on the initial conditions, there exists a small neighborhood of $x$ such that every point of it lies on exactly one of those orbits for a suitable $y_{0}$ and hence we associate to it the tangent to the orbit at that point, and this gives an extension of $\dot{x}$.

\begin{center}
\begin{figure}[H]
\centering
\includegraphics[scale=.4]{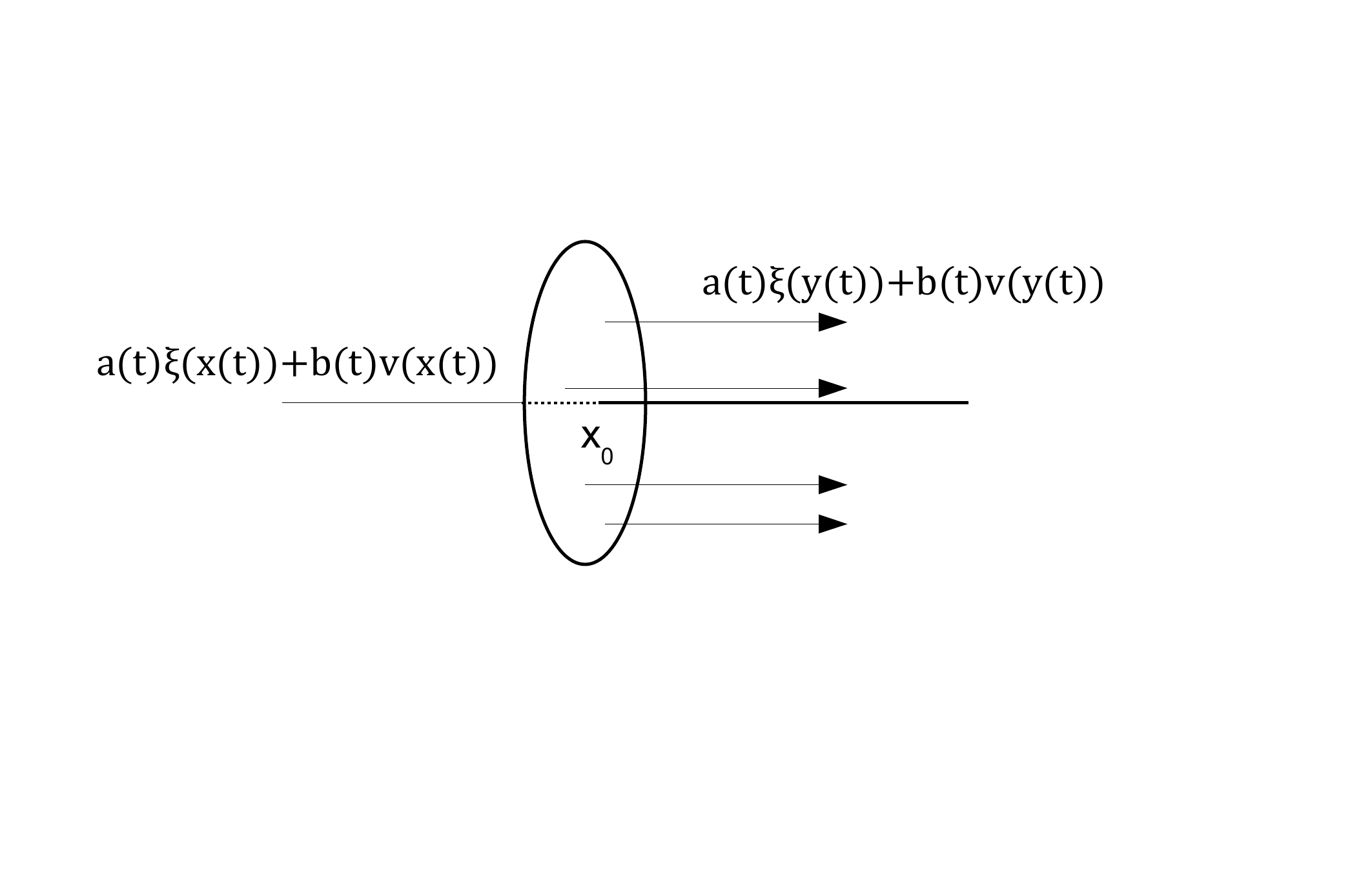}
\caption{Extension of the tangent vector}
\label{fig1}
\end{figure}
\end{center}

\noindent
It is important to notice that in fact this construction does not really give a vector field if the curve is not embedded. If the curve self-intersects then there could be a problem, but notice that the transport map is well defined by use of the dynamical system above.

\subsection{First deformation and formation of Dirac masses}

\noindent
Let us consider a compact set of curves in $\mathcal{L}_{\beta}$, which we can view as given by a map
$$f:S^{l}\mapsto \mathcal{L}_{\beta}$$
for some $l\in \mathbb{N}$. We claim (see section 2 ) that after a small $C^1$-perturbation, we may assume that all curves in $K=f(S^{l})$ have $\dot{x}=a\xi +bv$, with $a$ and $b$ having a finite number of zeros. We will use this result only for the case when $a$ has a finite number of zeros.\\
Hence curves in the class $K$ can be seen as pieces of curves with $a>0$ and pieces with $a<0$. In order to start our argument, we consider the case when the the curve has only two distinct zeros at $t_{0}$ and $t_{1}$. The argument will be extended later using a filtration adapted to the number of zeros of $a$. The filtration will include degenerate zeros of various orders.\\
Let us consider a curve $x(t)$ under the assumptions described above and assume that $a$ is negative in $(t_{0},t_{1})$. We will lift in the sequel the pieces with $a<0$ by using the transport map $\varphi_{t}$. Let us now describe this procedure.\\
Let $t_{2}=\frac{t_{0}+t_{1}}{2}$. Then, since $a(t_{2})\not =0$, we can define
$$ s_{0} = \inf \lbrace s>0; D \varphi_{s} ( \dot{x}(t_{2})) = \gamma \xi , \gamma >0 \rbrace.$$
Therefore if we consider the function $$g(s,t)=\varphi^{*}_{s}\beta_{x(t)}(\dot{x}(t))$$ then one has     $$g(s_{0},t_{2})=0$$ and
$$\partial_{s} g(s,t)= \varphi^{*}_{s} \left( \mathcal{L}_{v} \beta \right)(\dot{x}(t))$$
$$=d \beta_{\varphi_{s}(x(t))}(v,D \varphi_{s}(\dot{x}(t))).$$
Hence, at $(s_{0},t_{2})$, we get
$$\partial_{s} g(s_{0},t_{2}))=-\alpha(v,[v,\gamma \xi])=-\gamma<0.$$
Thus, by mean of the implicit function theorem, we can define the piece of curve $y(t)=\varphi_s(t)(x(t))$, for $t \in (t_{0},t_{1})$.
Notice that $$\dot{y}(t)=D \varphi_{s(t)}(\dot{x}(t))+\dot{s}(t) v=\gamma \xi+\dot{s}(t) v.$$
If we now close the curve by pieces of $v$, we transform our original curve $x$ to a curve $\tilde{x}$ in $\mathcal{L}_{\beta}$ having pieces with $a>0$, pieces with $a=0$ and some isolated zeros of $a$ (see figure \ref{fig2}).\\

\begin{center}
\begin{figure}[H]
\centering
\includegraphics[scale=.4]{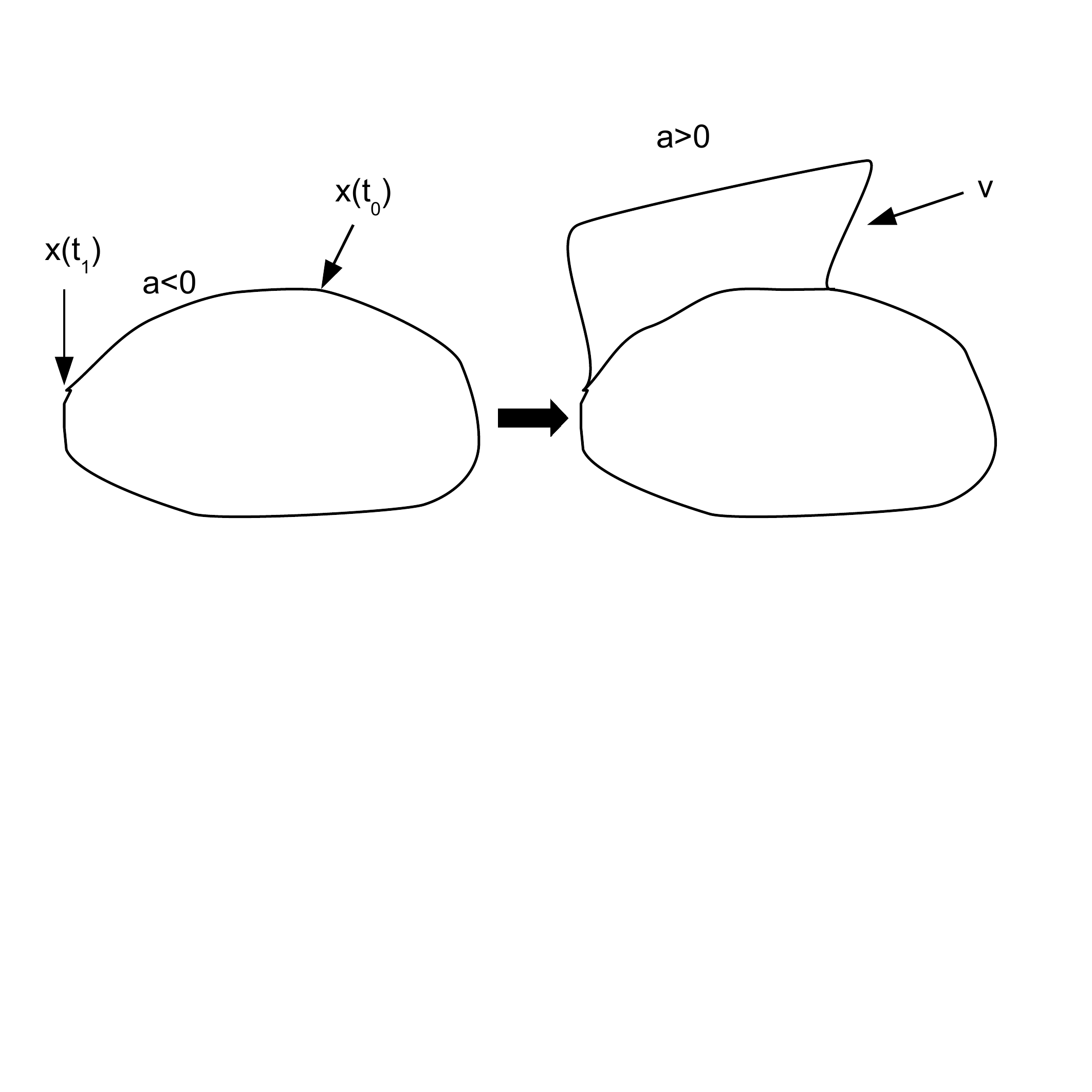}
\caption{Lifting the negative parts}
\label{fig2}
\end{figure}
\end{center}

\noindent
Notice that we can transform this process into a gradual process that will not take place in $\mathcal{L}_{\beta}$, that is, taking the map $f:S^l \longmapsto \mathcal{L}_{\beta}$ we construct a homotopy $U:[0,1]\times S^l \longmapsto \Lambda(M)$ such that $U(0,\cdot)=f(\cdot)$ and $U(1,\cdot)$ is valued in $\mathcal{C}_{\beta}^{+}$. Since the injection of $\mathcal{L}_{\beta}\hookrightarrow \Lambda(M)$ is a homotopy equivalence and since $\mathcal{C}_{\beta}^{+}$ injects into $\Lambda(M)$, this will lead us, after we resolve the issue of the ``Dirac masses'', to the fact that $\mathcal{C}_{\beta}^{+} \hookrightarrow \Lambda(M)$ is an $S^1$-homotopy equivalence. In fact, by using the previous construction we see that the new curve that we get in $\mathcal{C}_{\beta}^{+}$ reads as $y(t) = \varphi_{s(t)}(x(t))$, thus the homotopy that one might consider would be $H(t,l):=\varphi_{ls(t)}(x(t))$ (see figure \ref{fig3}).

\begin{center}
\begin{figure}[H]
\centering
\includegraphics[scale=.4]{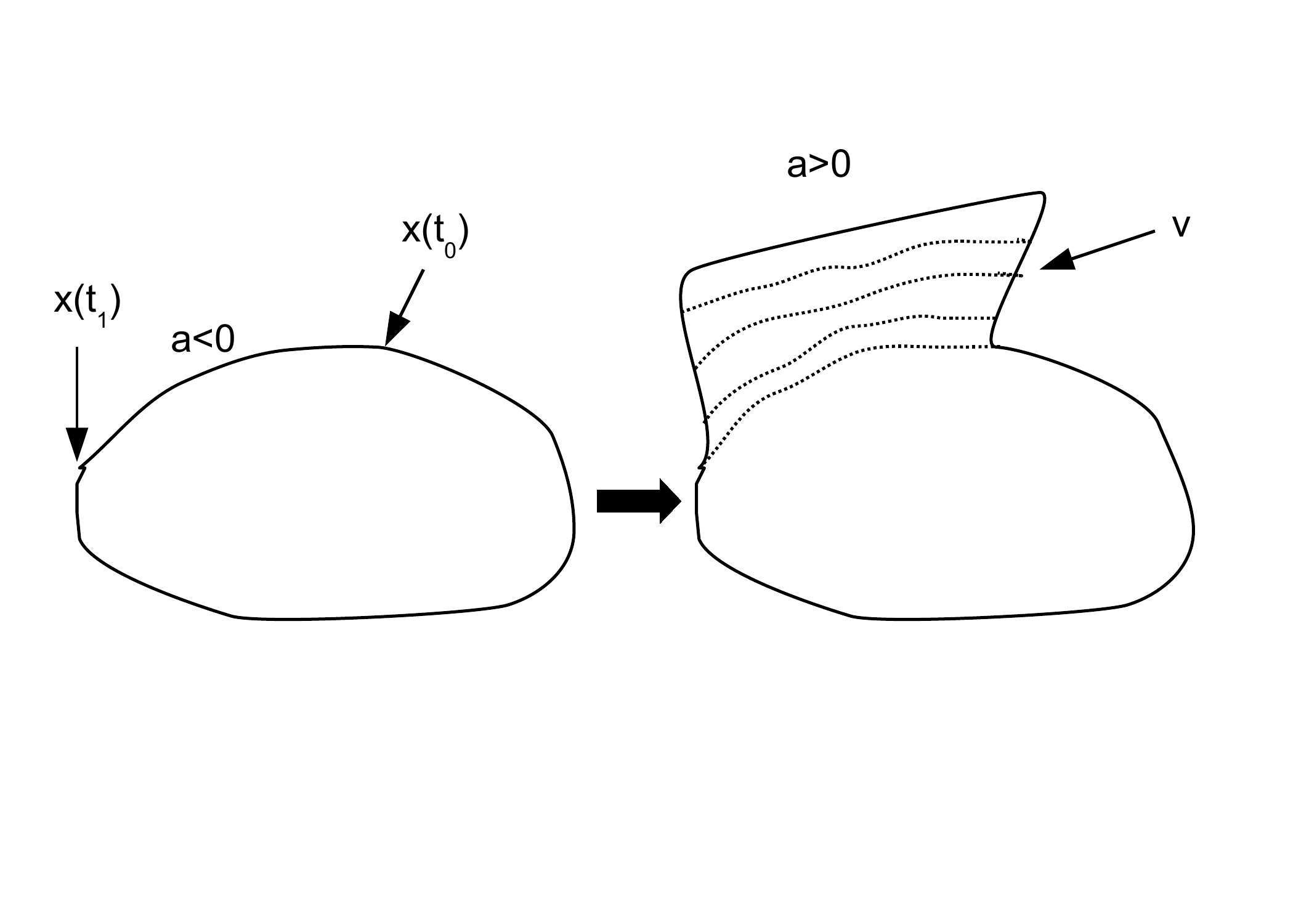}
\caption{The lifting as a deformation}
\label{fig3}
\end{figure}
\end{center}

\noindent
Now thinking of the base curve, the zeros $t_{0}$ and $t_{1}$ of $a$ can come to each other, collide and cancel as $x$ varies in $S^l$ and $f(x)$ varies in $K=f(S^l)$. Tracking our construction over this deformation we see how a Dirac Mass, that is a back and forth run along $v$, can be created as two zeros of $a$ come to each other and collapse (see figure \ref{fig4}).

\begin{center}
\begin{figure}[H]
\includegraphics[scale=.4]{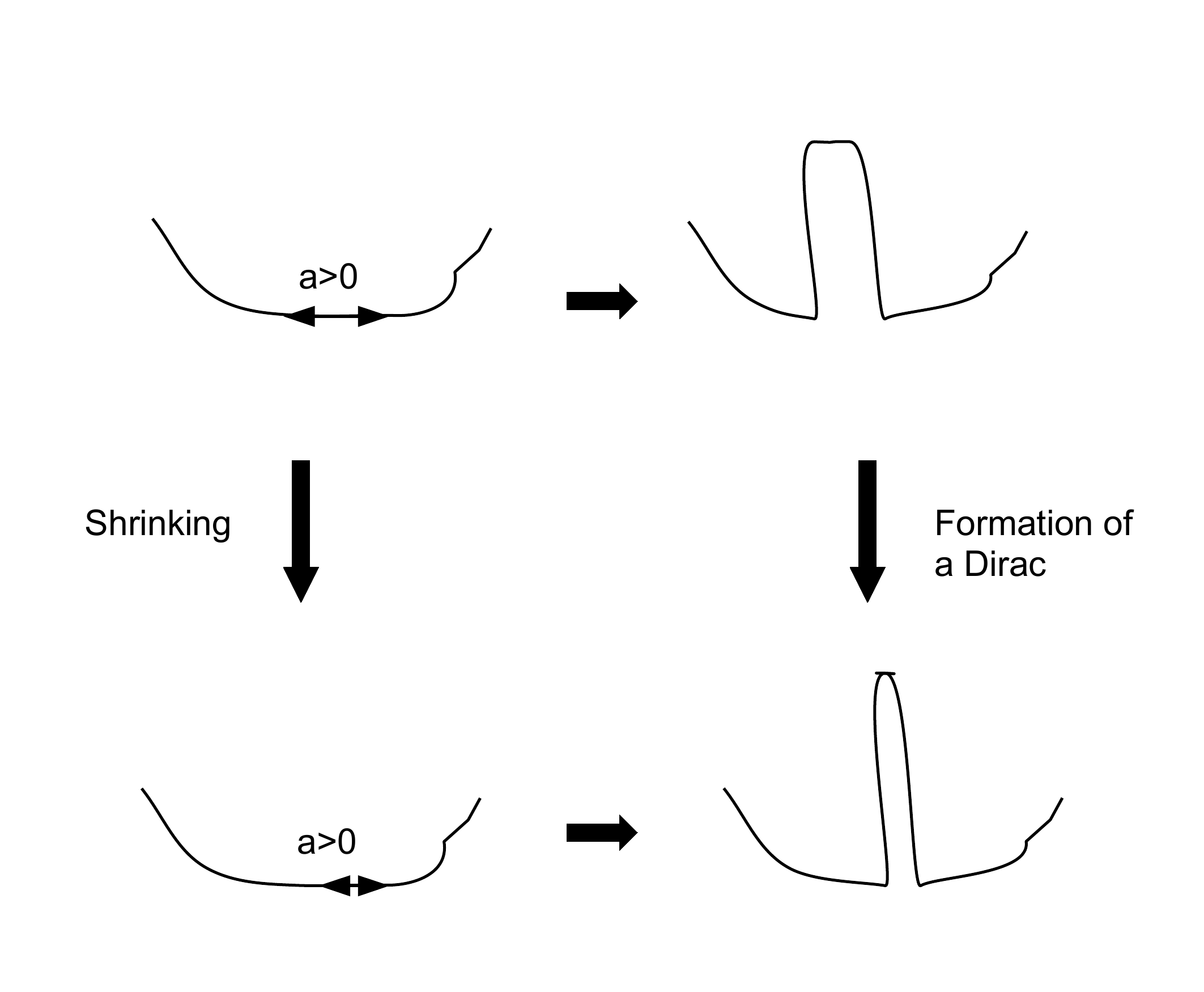}
\caption{Formation of a ``Dirac Mass''}
\label{fig4}
\end{figure}
\end{center}

\noindent
More complicated phenomena take place as we resolve the case of collapse of zeros of $a$ with higher multiplicity, also as $x$ varies in $S^l$. Let us understand first the collapse of two zeros.

\subsection{Simple Dirac mass}

\noindent
Here we consider a curve $x$, such that $\dot{x}=a\xi + bv$ where $a$ is positive everywhere except at a point $x(t_{0})$ where there is a back and forth $v$ jump of length $l$.\\
Consider the family $y(t,s)$ for $s\in [0,1]$, that coincides with the curve $x$ and at $x(t_{0})$, $y(t,s)$ has a back and forth $v$ jump of length $sl$. By using a lemma in \cite{B1}, we recognize a process with which the Dirac mass can be gradually removed (see figure \ref{fig5}).

\begin{center}
\begin{figure}[H]
\centering
\includegraphics[scale=.4]{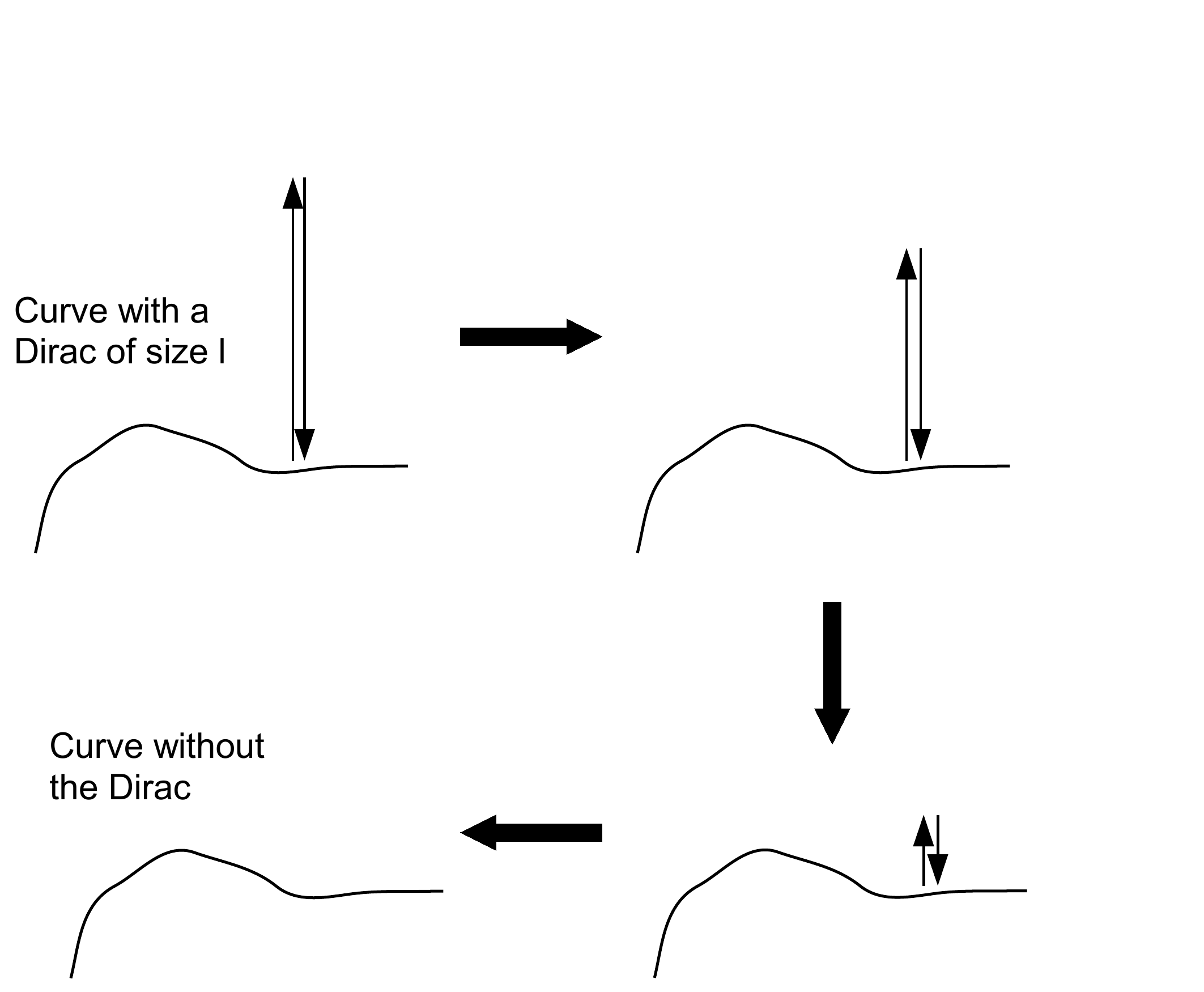}
\caption{Cancellation of a Dirac mass}
\label{fig5}
\end{figure}
\end{center}

\noindent
\subsection{Cancellation Process}

\noindent
Here, let us consider a curve close to $x$, in the case it supports a ``nearly'' Dirac mass, namely the former back and forth run along $v$ is now ``opened'' by a little bit: a small curve of length $\varepsilon$, where $\dot{x}=a\xi +bv$ and $a>0$ is inserted in between the forth and back run along $v$.\\
To remove this nearly Dirac mass, with a process that coincides with the one in the previous subsection as $\varepsilon$ goes to zero, we construct a deformation vector $Z=\lambda \xi +\mu v +\eta [\xi,v]$ along the curve. This is done by first taking $-v$ at B, in fact, if we want it to be adapted to the length of the size of the Dirac mass, then we should take $-lv$ instead at B, where $l$ is the length of the $v$ jump. We will disregard this fact for now. So, after taking $-v$ at B, we transport it along $-\dot{x}$ ($\dot{x}$ is close to $\xi$ on the ``opening'' so that the extension described above is possible) till we reach A. Eventually, we will have a non-zero $[\xi; v]$ component for the transported vector at A. We then transport, by using the transport map of $v$, our vector from A to C and we adjust the $v$-length of this $v$-jump, adding or subtracting a $\delta s v$, so that the $v$ component of the transported vector at C is zero. At C, we know that $a=0$ and this is inconvenient. Therefore, we transport it a bit further to a point $p$ where $a,|b| > c > 0$,(see figure \ref{fig6}).

\begin{center}
\begin{figure}[H]
\includegraphics[scale=.4]{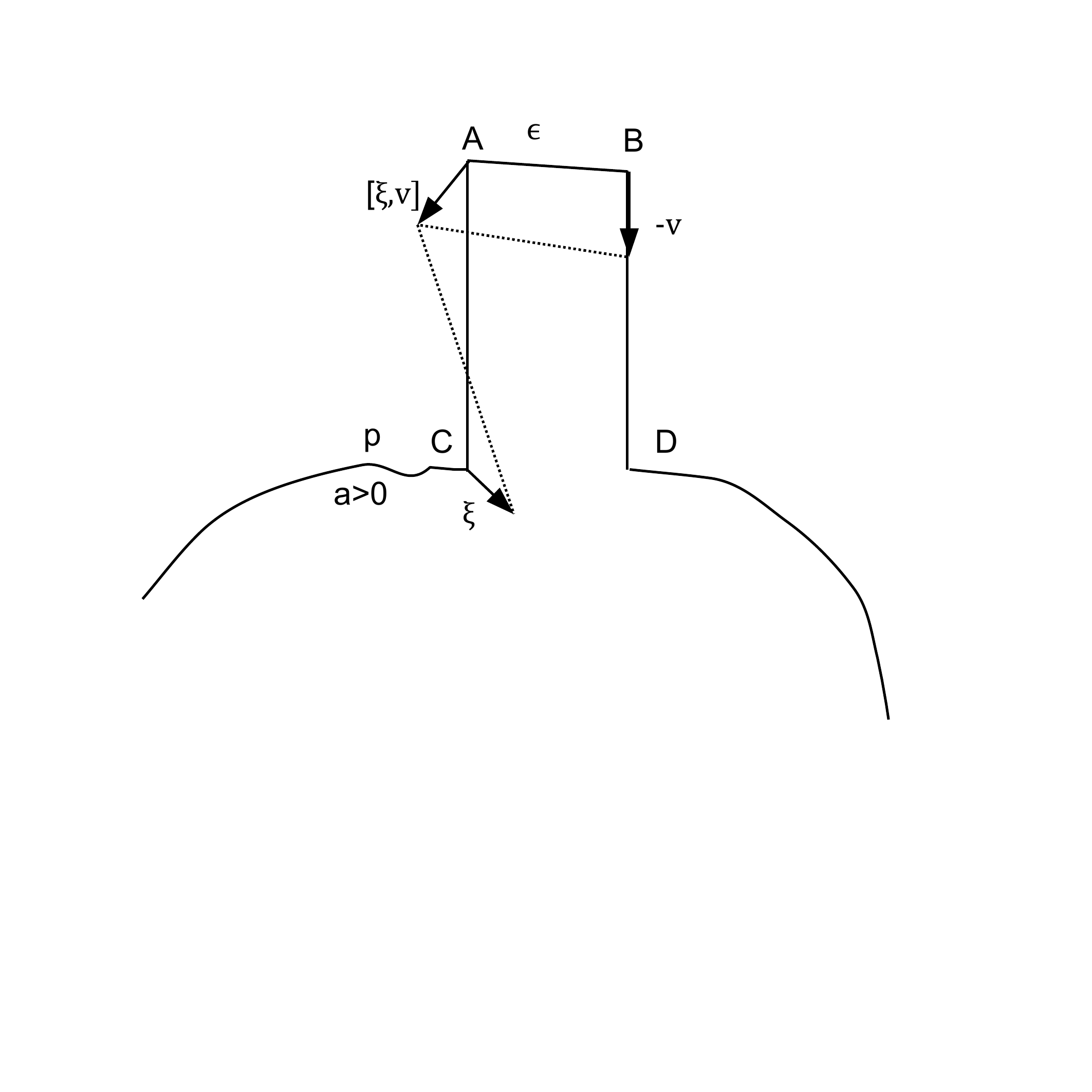}
\caption{Nearly Dirac mass}
\label{fig6}
\end{figure}
\end{center}

\noindent
The requirement that $|b|>c$ at $p$ is not needed, but will be used to prove a stronger result of deformation, along which the number of zeros of $b$ does not increase over the deformation. It can be dropped for the proof of Theorem 1. We need now to compensate the vector that we got at $p$ and to derive precise estimates on $a$ and $b$ as they are deformed to adjust the variation of the curve induced by the tangent vector described above.\\
The main idea of the compensation is first to span the kernel of $\alpha$ at $p$. This will be done by using a combination of two process involving the introduction of a small perturbation to the curve.\\
The first process is meant to generate a $v$ component at a given point on the curve. Namely, given a point $x_{1}$ on a curve $x$ of $\mathcal{C}_{\beta}$ (in fact we only need that the point $x_{1}$ is located in a portion of the curve with $a>0$) we can construct a small variation of the curve near $x_{1}$ so that we get a vector almost equal to $v$ at that point. In fact, this can be described as changing $b$ to $\lambda b$, in a small interval before $x_{1}$, with $\lambda>1$ to generate $v$ and $\lambda<1$ to generate $-v$ (see figure \ref{fig7}).

\begin{center}
\begin{figure}[H]
\centering
\includegraphics[scale=.4]{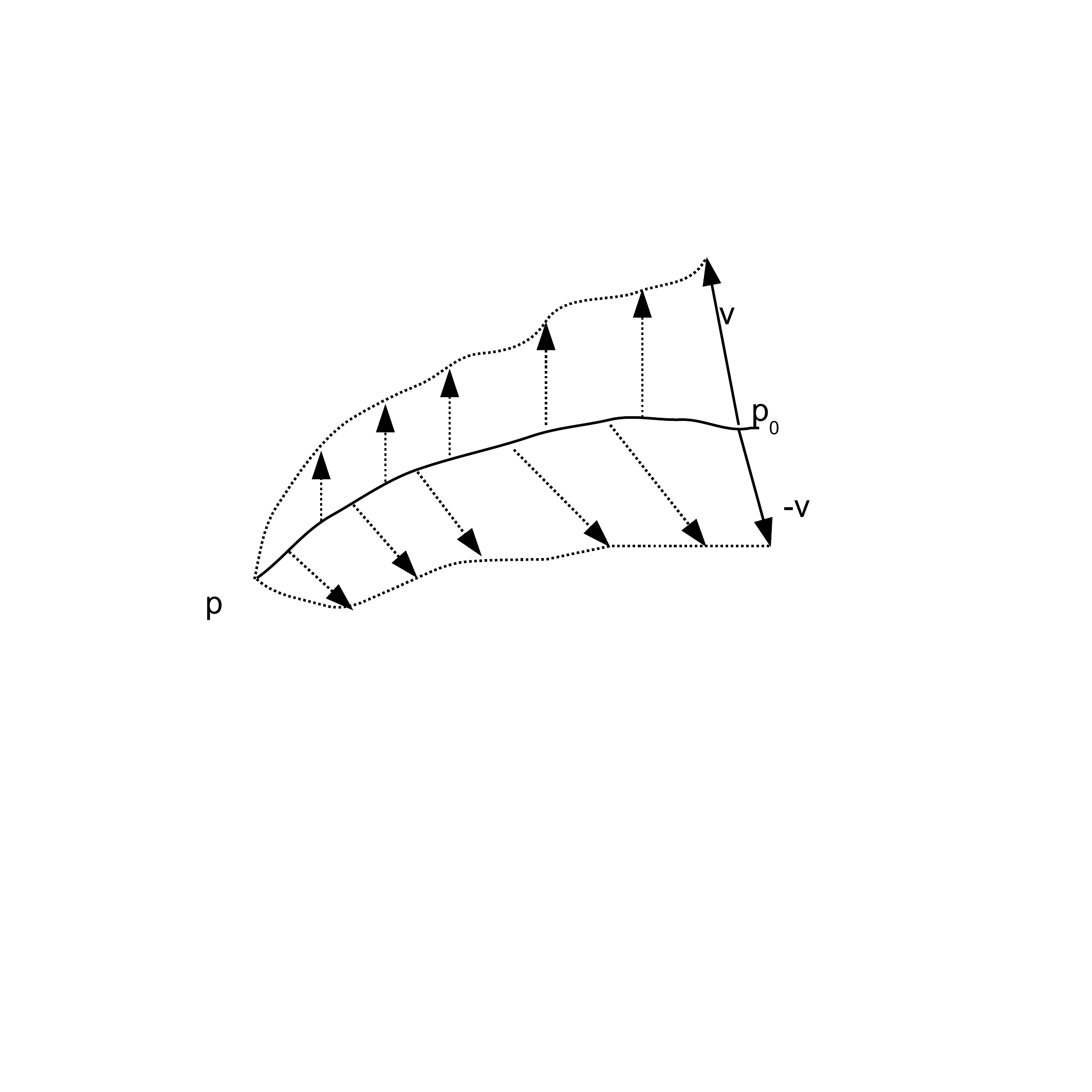}
\caption{Perturbation of $b$ to create $v$}
\label{fig7}
\end{figure}
\end{center}

\noindent
The second step consists on generating a $[\xi,v]$ component at a given point of a curve $x$ in $\mathcal{C}_{\beta}$. This will be done by transporting $v$ (and this is where the combination comes in) along $\dot{x}$ to the point $x_{2}$ where we want to get the $[\xi,v]$ component (see figure \ref{fig7.1}); of course we will have other components at that point but we will show that they have a minor contribution.

\begin{center}
\begin{figure}[H]
\centering
\includegraphics[scale=.8]{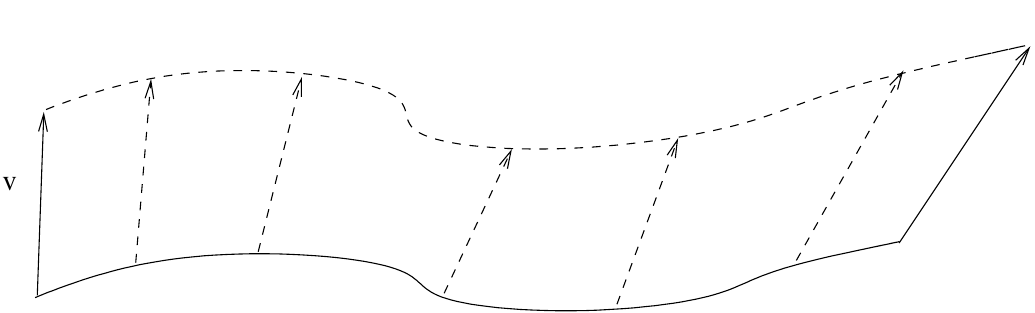}
\caption{Creation of a $[\xi,v]$ component}
\label{fig7.1}
\end{figure}
\end{center}


\noindent
We will show that a combination of these two steps, after a careful choice of the points in the curve for inserting them, indeed span the kernel of $\alpha$ (after a projection parallel to $\dot{x}$) at the desired point. The last step then consists of removing the undesired components along $\xi$. This will be done by using the transport map along the curve as described in paragraph 3.1, as we will transport the vector $a\xi+bv$ along itself, and by adjusting the length, we can cancel the additional component to find a resultant vector in kernel of $\alpha$.\\
All the previous construction will be made precise as it depends tightly on the choice of the points and on the portions of the curve on which the deformation be built (see figure \ref{fig8}).

\begin{center}
\begin{figure}[H]
\centering
\includegraphics[scale=.6]{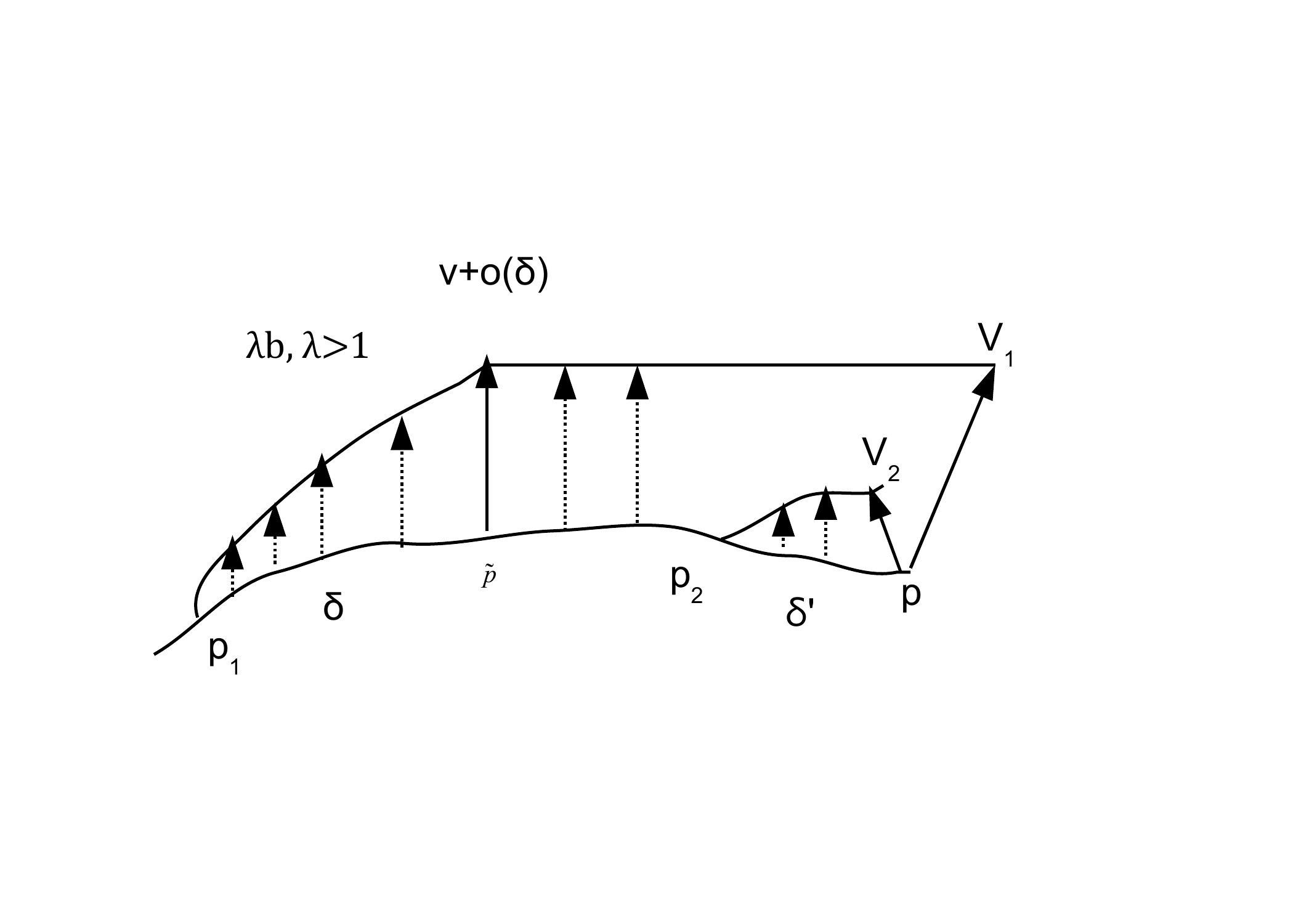}
\caption{Combination of the processes}
\label{fig8}
\end{figure}
\end{center}

\noindent
We will find a vector $Z=\lambda \xi +\mu v +\eta [\xi,v]$ which is locally Lipschitz in the $\mathcal{H}^{2}$ topology in a neighborhood of the curve $x$ (that is the topology of $\mathcal{L}_{\beta}$). It is important to follow the full construction: we start from a compact set $K$ in $\mathcal{L}_{\beta}$ endowed with the $\mathcal{H}^{2}$ topology and we deform it, by using the regularizing flow in section 2, to a compact set $\tilde{K}$ of smooth curves. This construction that follows is done on curves in $\tilde{K}$, and it is done curve by curve. It will be made continuous by using a partition of unity adapted to a covering of $\tilde{K}$ in the $\mathcal{H}^{2}$ topology: these details are in the Appendix.\\
Let us start now by detailing the construction that will be carried on four steps as mentioned above.\\
Let us consider the following system of differential equations:

\begin{equation}\label{equa 1}
\left\{\begin{array}{lll}
\dot{\lambda}=b\eta\\
\\
\dot{\mu}=(b\overline{\mu}_{\xi}-a\tau)\eta +h_{\delta}\\
\\
\dot{\eta}=b\overline{\mu}\eta +\mu a-\lambda b
\end{array}
\right.
\end{equation}

\noindent
The first equation of this system tells us that $a = \alpha(\dot{y})$ remains constant along the variation. The second equation provides us the variation of $b$. The third equation tells us that the curve stays in the space $\mathcal{L}_{\beta}$. In fact the second equation corresponds to the introduction of a small variation of $b$ using the function $h_{\delta}$ that will be chosen depending on the situation, also the initial conditions will vary with the different cases of $h_{\delta}$ that we will consider.

\subsubsection{Case $h_{\delta}\not = 0$}

\noindent
We will give here the details on how to generate the vector $v$ at a point as it was mentioned above.\\
In what follows for every $\delta > 0$, $h_{\delta}$ denote a positive function, with compact support in $(0, \delta)$ such that $h_{\delta}= 1$ in the interval $[\frac{\delta}{4},\frac{3\delta}{4}]$. We will first describe the process that will be used starting from any point $x_{0}$ in the curve, and then insert them at specific points to get the construction described above.\\
Consider the following system:

\begin{equation}\label{equa 1}
\left\{\begin{array}{llll}
\dot{\lambda}=b\eta\\
\\
\dot{\mu}=(b\overline{\mu}_{\xi}-a\tau)\eta +h_{\delta}\\
\\
\dot{\eta}=b\overline{\mu}\eta +\mu a-\lambda b\\
\\
\lambda(0)=\eta(0)=\mu(0)=0
\end{array}
\right.
\end{equation}

\noindent
As mentioned before the second equation provides us the variation of $b$. So here this will induce a change of $b$ in the positive direction. We solve this system in the interval $[0; 2\delta]$ so that there is no zero of $b$ in this interval, i.e. $\frac{1}{c} < b < c$; this is always possible since $b$ is real analytic and it is possible for all the curves in an $\mathcal{H}^{2}$ neighborhood of the curve (see section 2).\\
A study of the previous system leads us the following lemma that will be proved in the rest of this paragraph:

\begin{lemma}
Let us assume that $Z_{1}$ satisfies (\ref{equa 1}), then:

\begin{equation}
\left\{ \begin{array}{lll}
\|\lambda \|_{\infty}\leq C(\|a\|_{\infty}+\|b\|_{\infty})\delta^{2}\\
\\
\|\mu-\int_{0}^{t}h_{\delta}(s)ds\|_{\infty}\leq C(\|a\|_{\infty}+\|b\|_{\infty})\delta^{2}\\
\\
\|\eta\|_{\infty}\leq C(\|a\|_{\infty}+\|b\|_{\infty})\delta^{2}
\end{array}
\right.
\end{equation}
\end{lemma}

\begin{proof}
Let $A$ be the matrix of the previous differential system, that is
$$A=\left(\begin{array}{ccc}
0&0&b\\
0&0&(b\overline{\mu}_{\xi}-a\tau)\\
-b&a&b\overline{\mu}
\end{array}
\right),$$
and let $R$ denote the resolvent of the system, that is $\dot{R}=AR$ with $R(0)=id$. Then, if
$Z_{1}=\left(\begin{array}{ccc}
\lambda\\
\mu\\
\eta
\end{array}
\right)$,
we have
$$Z_{1}(t)=R(t)\int_{0}^{t}R^{-1}(s)\left(\begin{array}{ccc}
0\\
h_{\delta}(s)\\
0
\end{array}
\right)ds.$$

\noindent
Now since $\dot{R}=AR$ we have that
\begin{equation}
|R(t)-id|\leq \int_{0}^{t}|A(s)|ds+\int_{0}^{t}|A(s)||R(s)-id|ds,
\end{equation}
It follows from Gronwall's lemma that
\begin{equation} \label{est}
|R(t)-id|\leq  \int_{0}^{t}|A(s)|ds e^{\int_{0}^{t}|A(s)|ds},
\end{equation}
Hence for $t\in[0,\delta]$ we have $$|R(t)-id|\leq C \int_{0}^{t}|A(s)|ds .$$
Then
$$|Z_{1}(t)-\left(\begin{array}{ccc}
0\\
\int_{0}^{t}h_{\delta}(s)ds\\
0
\end{array}
\right)| \leq C\int_{0}^{t}\int_{0}^{s}|A(u)|du ds.$$
Therefore we deduce the result of the lemma.
\end{proof}

\noindent
According to the previous lemma we can estimate the change of $b$ and $a$ along the deformation introduced by the vector field $Z_{1}$ above. Knowing that, once extended, the evolution equations of $a$ and $b$ read as
$$\frac{\partial a}{\partial s}=\dot{\lambda}-b\eta$$
and
$$\frac{\partial b}{\partial s}=\dot{\mu}+(a\tau - \overline{\mu}_{\xi}b)\eta.$$

\noindent
We get that $a$ is unchanged and after a bootstrapping argument (see Appendix) we have $$|b(s,t)-b(t)|\leq Cs|h_{\delta}| \leq Cs.$$

\subsubsection{Case $h_{\delta}=0$}

\noindent
We now consider the same system of equations, but with $h_{\delta}= 0$ and with initial conditions non-zero, that is:

\begin{equation}\label{equa 2}
\left\{\begin{array}{llll}
\dot{\lambda}=b\eta\\
\\
\dot{\mu}=(b\overline{\mu}_{\xi}-a\tau)\eta \\
\\
\dot{\eta}=b\overline{\mu}\eta +\mu a-\lambda b\\
\\
\lambda(0)=\eta(0)=0,\; \mu(0)=1
\end{array}
\right.
\end{equation}

\noindent
This will allow us to generate a non-trivial $[\xi,v]$ component at the point $p$, with of course an extra term $r$ that needs to be removed in a later stage.\\
If $Z_{2}$ is a solution of this equation then we have $Z_{2}(t)=R(t)Z_{2}(0)$. Notice now that
$$|\big(R(t)-id-\int_{0}^{t}A(s)ds\big)Z_{2}(0)|=|\int_{0}^{t}A(s)(R(s)-id)dsZ_{2}(0)|$$
Using the estimate (\ref{est}) we have
$$|Z_{2}(t)-Z_{2}(0)-\int_{0}^{t}A(s)Z_{2}(0)|\leq C \delta^{2}(||a||_{\infty}+||b||_{\infty})^{2}$$
Therefore we have
$$\|\eta-\int_{0}^{t}a(s)ds\|_{\infty}\leq C\delta^{2}(\|a\|_{\infty}+\|b\|_{\infty})^{2}$$
and
$$\|\mu-1\|_{\infty}\leq C \delta^{2}(\|a\|_{\infty}+\|b\|_{\infty})^{2}$$
We will set
$$\theta_{\delta}=\delta^{2}(\|a\|_{\infty}+\|b\|_{\infty}), \qquad \tilde{\theta}_{\delta}=\delta^{2}(\|a\|_{\infty}+\|b\|_{\infty})^{2}$$

\subsubsection{Combination}

\noindent
Now we will use a combination of these processes starting at specific points on the curve to span the kernel of $\alpha$ at $p$. So here, given a point $p$ on the curve we will use 3 points $p_1$, $\overline{p}$ and $p_2$, and we re-parametrize our curve so that zero corresponds to the point $p_1 = x(0)$ whereas the time $\delta$ will correspond to the point $\overline{p}$, and take $p_2 = x(2\delta-\delta')$, $p = x(2\delta)$, where here $0 < \delta' << \delta$. We will provide more details about the values of $\delta$ and $\delta'$ in the sequel. Also $\delta$ does not need to be small.\\
From $p_1$ we use the construction done with (\ref{equa 1}) up to time $\delta$. Then again, use the process described by (\ref{equa 2}) starting from $\overline{p}$ with initial condition the resultant vector from the first construction, till we reach the point $p$. And to finish we run again the first process (that is using (\ref{equa 1})) starting from $p_2$ till we reach $p$ (see figure \ref{fig8}).\\
Let us see what are the vectors formed now at the point $p$. From the first and the second process we get a vector $$V_{1}=\int_{0}^{\delta}h_{\delta}(s)ds
\Big[\big(1+O(\tilde{\theta}_{\delta})+\frac{b}{a}O(\theta_{\delta})\big)v+$$
$$+\big(\int_{0}^{\delta}a(s)ds+O(\tilde{\theta}_{\delta})\big)[\xi,v]+\frac{1}{a}O(\theta_{\delta})(a\xi+bv)+
O(\delta \theta_{\delta})\Big]$$

\noindent
and from the third process, we have $$V_{2}=O(\theta_{\delta'})(a\xi+bv)+\big(\int_{0}^{\delta'}h_{\delta'}(s)ds+
O(\theta_{\delta'})\big)v+O(\theta_{\delta'})[\xi,v]$$

\noindent
Now we compute the determinant $det(P(V_1), P(V_2))$, where $P$ is the projection, on $\ker\alpha$, parallel to $a\xi + bv$, we find:
$$\left\vert \begin{array}{cc}
1+O(\tilde{\theta}_{\delta})+\frac{b}{a}O(\theta_{\delta}) & O(\tilde{\theta}_{\delta})\\
\\
\int_{0}^{\delta'}h_{\delta'}(s)ds+O(\theta_{\delta'}) & O(\theta_{\delta'})
\end{array}
\right\vert$$
The dominant term of this determinant is $$O(\theta_{\delta'})-\int_{0}^{\delta'}h_{\delta'}(s)ds\int_{0}^{\delta}a(s)ds=O(\theta_{\delta'})-\delta \delta' a_{0} +o(\delta \delta')$$
Since $\delta'<<\delta$ this determinant is bounded away from zero.\\
Now the global estimate on $b$ after extension the of the deformation vector $Z$, reads as follows:
$$|b(s,t)- b(t)|\leq Cs|h_{\delta}(t) + h_{\delta'}(t)|\leq Cs.$$

\subsubsection{Compensation of $\xi$}

\noindent
Notice that now the only part that needs compensation is the $\xi$ component. Since we extended the velocity vector of the curve to some small $\mathcal{H}^{2}$ neighborhood of the curve, by transporting $a\xi +bv$ from $p_2$, we get a non-zero $\xi$ component at $p$. Notice that this corresponds to the use of the transport map $\phi_{t(x)}(x)$ where here $t(x)$ is the necessary time to be able to compensate the given $\xi$ component. This can be made precise if we get the right estimates on the transported vector from B. Let $S$ be the section at $p$ of $ker_{\alpha}$ and $S_{2}$ a section of $\ker_{\alpha}$ at $B$. We consider also the section $\tilde{S}_{2}=D\phi_{t}(S_{2})$ the image of the section $S_{2}$ under the diffeomorphism $\phi_{t_{0}}$ where $t_{0}$ is the necessary time to reach $p$ starting from $B$. Now we want to find a way of projecting the section $\tilde{S}_{2}$ on $S$ using the diffeomorphism $\phi_{t}$. In fact, we have
$$D(\phi_{t}(p))(\cdot)=D\phi_{t}(\cdot)+dt(\cdot) (a\xi+bv)$$
evaluating at $t=0$, the previous equation reads as
$$D(\phi_{t}(p))(X)=X+dt(X)(a\xi+bv)$$
for every $X\in T_{p}M$. Therefore we can always project on $S$ by taking $dt(X)=\frac{\alpha(X)}{a}$, noticing that $dt(X)=0$ means we are already in $S$, and if $dt(X)\not =0$ then by taking $\phi_{st}(p)$ and adjusting the $s$ we can always compensate the $\xi$ component. The same procedure can be done for the section spanned by the vectors $V_{1}$ and $V_{2}$ at $p$ and projecting them on $S$ to get components free from $\xi$.\\
Now one needs to estimate the size of the component that needs to be compensated since the previous procedure corresponds to an increase or decrease in time. Hence it will change the parametrization of our curve.

\begin{center}
\begin{figure}[H]
\centering
\includegraphics[scale=.4]{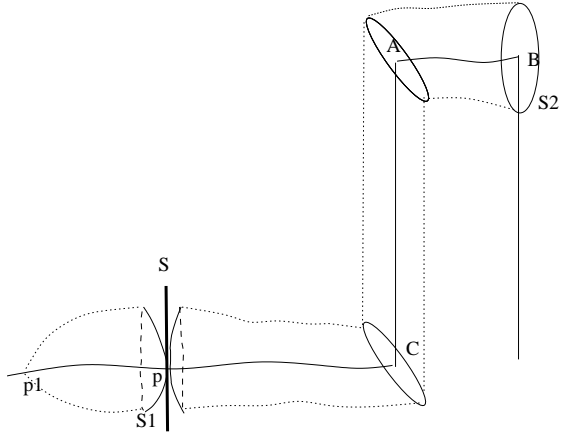}
\caption{Compensation of $\xi$}
\label{fig9}
\end{figure}
\end{center}

\noindent
Let $\varepsilon$ be the opening of the nearly Dirac mass. Since we are transporting the vector $-v$ starting from B, the transport equation is equivalent to solving

\begin{equation}\label{equa 3}
\left\{\begin{array}{llll}
\dot{\lambda}=b\eta\\
\\
\dot{\mu}=(b\overline{\mu}_{\xi}-a\tau)\eta \\
\\
\dot{\eta}=b\overline{\mu}\eta +\mu a-\lambda b\\
\\
\lambda(0)=\eta(0)=0, \; \mu(0)=-1
\end{array}
\right.
\end{equation}

\noindent
This last system behaves as (\ref{equa 2}), starting from the point $\tilde{p}$. Thus, it holds
$$\mu=-1+O(\tilde{\theta}_{\varepsilon}), \qquad \eta=\int_{0}^{\varepsilon}a(s)ds+O(\tilde{\theta}_{\varepsilon})$$
Since the transport equation is linear, we have that at $p$ all the components of the transported vector are $O(\varepsilon)$ and so, $|dt|=O(\varepsilon)$. Notice now that if the initial length is $l$ then the new length will be $l'=l+O(\epsilon)$ and therefore the rescaled $b$ is $\tilde{b}=\frac{b}{1+O(\epsilon)}$. Again this gives a final estimate on $b$ along the variation as follows:
\begin{equation}\label{est2}
|b(s,t)-b(t)|\leq CsO(\epsilon).
\end{equation}

\begin{proposition}
There exists $\varepsilon_{0} > 0$ such that if the opening of the nearly Dirac mass is $\varepsilon < \varepsilon_{0}$, then the nearly Dirac mass can be gradually removed.
\end{proposition}

\begin{proof}
Recall that from the previous construction, $b$ will only change in the portion $[0, 2\delta]$ between $p_1$ and $p$. In that region we have that $\frac{1}{c} < b < c$, hence from the estimate (\ref{est2}) we have
$$\frac{1}{c}-Cs\varepsilon \leq b(s,t) \leq c+Cs\varepsilon.$$
Therefore, given a nearly Dirac mass of length $l$, if we take $\varepsilon < min(cCl, \frac{1}{2cCl})$ we find that $$\frac{1}{2c}\leq b(t,s)\leq 2c$$
for $s\in [0, l]$. Thus the process can be completed and the nearly Dirac mass can be removed with a control on the number of zeros of $b$.
\end{proof}

\noindent
After this compensation is done, we can see that this process will cancel the nearly Dirac mass, in fact if we let $Z$ the deformation vector built in the previous construction, then if we start by $-v$ at B, it is enough to check the behaviour of $\int b$. We have
$$Z\cdot \int b = \int \dot{\mu}+(a\tau-\overline{\mu}_{\xi})\eta$$
By splitting the integral into two pieces we see that:\\

from D to B we have $\eta=0$ hence $Z\cdot \int_{[DB]}b=-1$; \\

from B to A we have $Z\cdot b =0$ hence  $Z\cdot \int_{[BA]}b=0$.

\begin{proposition}
Let $l$ be the length of the nearly Dirac mass, then if $l$ tends to zero, the deformation tends to the identity.
\end{proposition}

\begin{proof}
One has to notice that, the previous construction was made regardless of the length of the nearly Dirac mass, and this can be adapted: instead of transporting $-v$ from the point B, we start by transporting $-lv$. Since the deformation was made using linear differential equations, one has that the new deformation vector is $\tilde{Z}(x)=l(x)Z(x)$, hence if $l$ tends to zero, the deformation tends to identity.
\end{proof}

\subsection{Case of a double zero}

\noindent
In this case we will consider two nearly Dirac masses, that is $3$ $v$-pieces, that might converge to a single jump. First thing to notice is that we can do our construction and build the deformation vector in two different ways, but we can convex combine them since they have independent supports assuming that the length of each intermediate piece is less that the $\varepsilon_{0}$ that we took in the case of a single nearly Dirac mass.

\begin{center}
\begin{figure}[H]
\centering
\includegraphics[scale=.4]{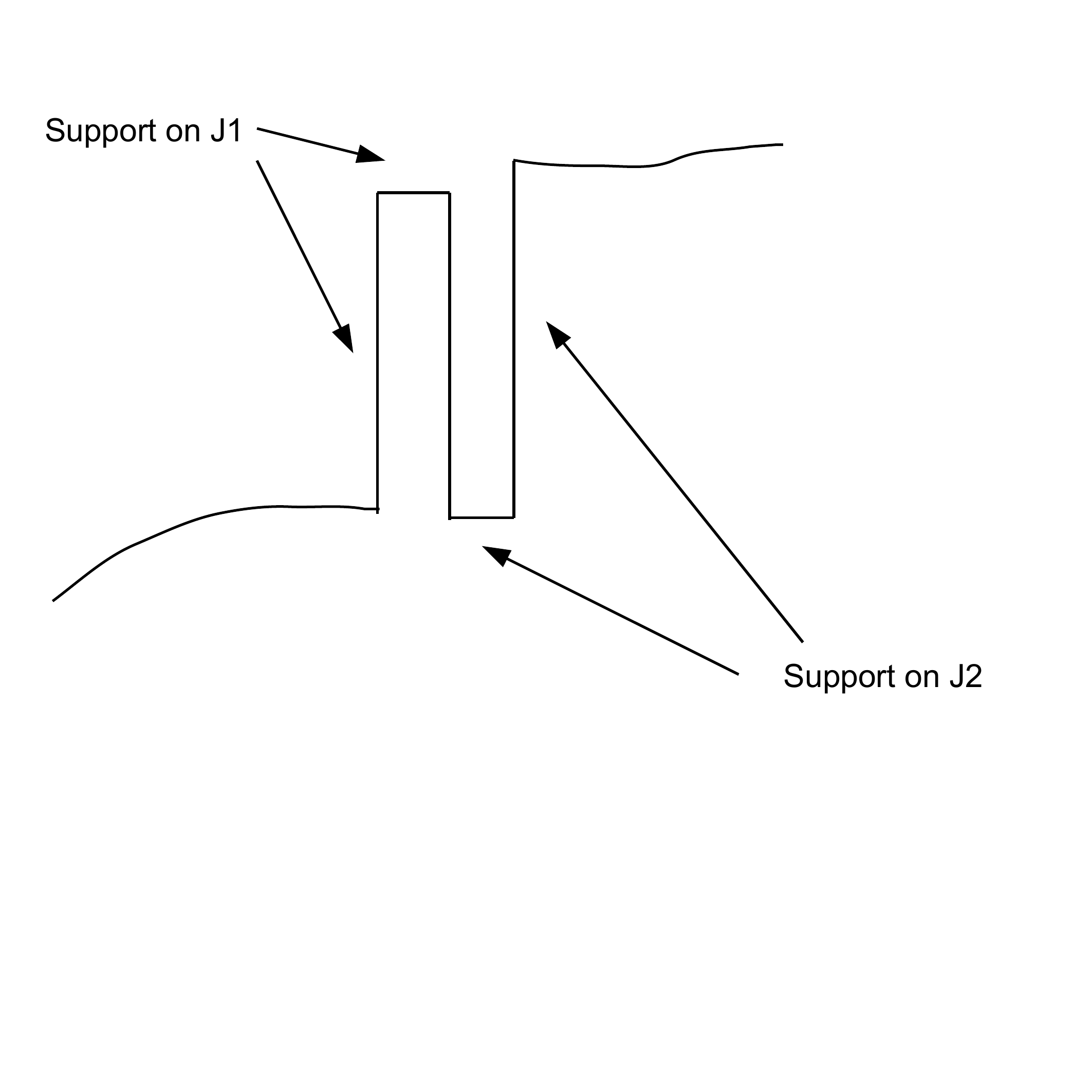}
\caption{Case of a double zero}
\label{fig10}
\end{figure}
\end{center}

\noindent
Hence, the two procedures can be run together without interfering, leading to a case where we have two positive (or negative jumps) linked by a piece of curve. Hence, we can convex combine them to end up with a step-like curve that moves along the convex combination between the two extremal parts that are curves
with a single $v$ jump.

\begin{center}
\begin{figure}[H]
\centering
\includegraphics[scale=.4]{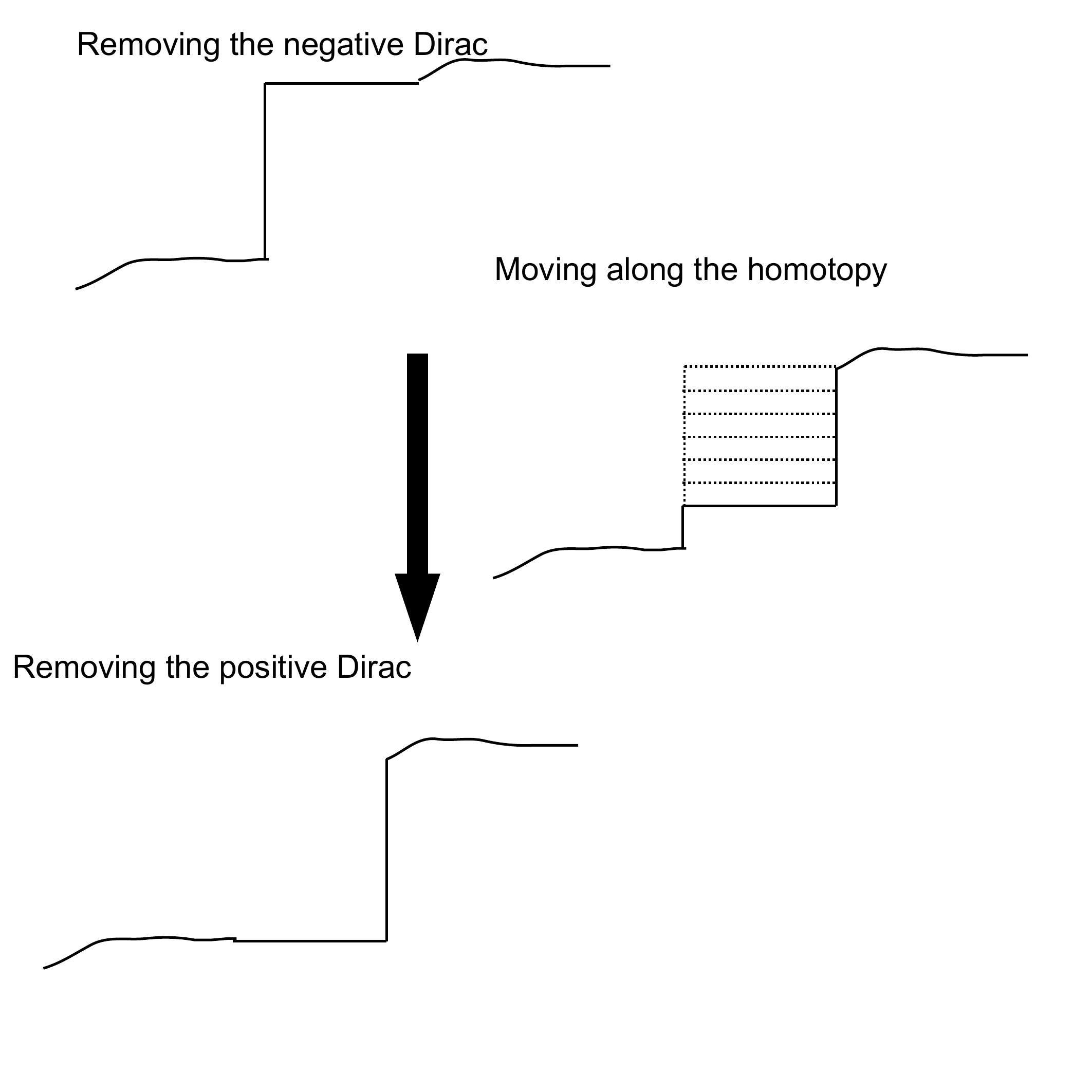}
\caption{Convex combination of the two process}
\label{fig11}
\end{figure}
\end{center}

\noindent
Another way to do this (which will be useful in the case of large multiplicity) is to build the deformation vector starting from one nearly Dirac mass and then crossing the other to finish the compensation from the other side (see figure \ref{fig12}).\\
In this case we need $\varepsilon_{1}+\varepsilon_{2}<\varepsilon_{0}$ and since the construction can be made from
both sides, they can be superposed since in the common support, it is just a transport equation that conserves all the quantities and hence they can be convex combined to get the same result as mentioned above.

\begin{center}
\begin{figure}[H]
\centering
\includegraphics[scale=.4]{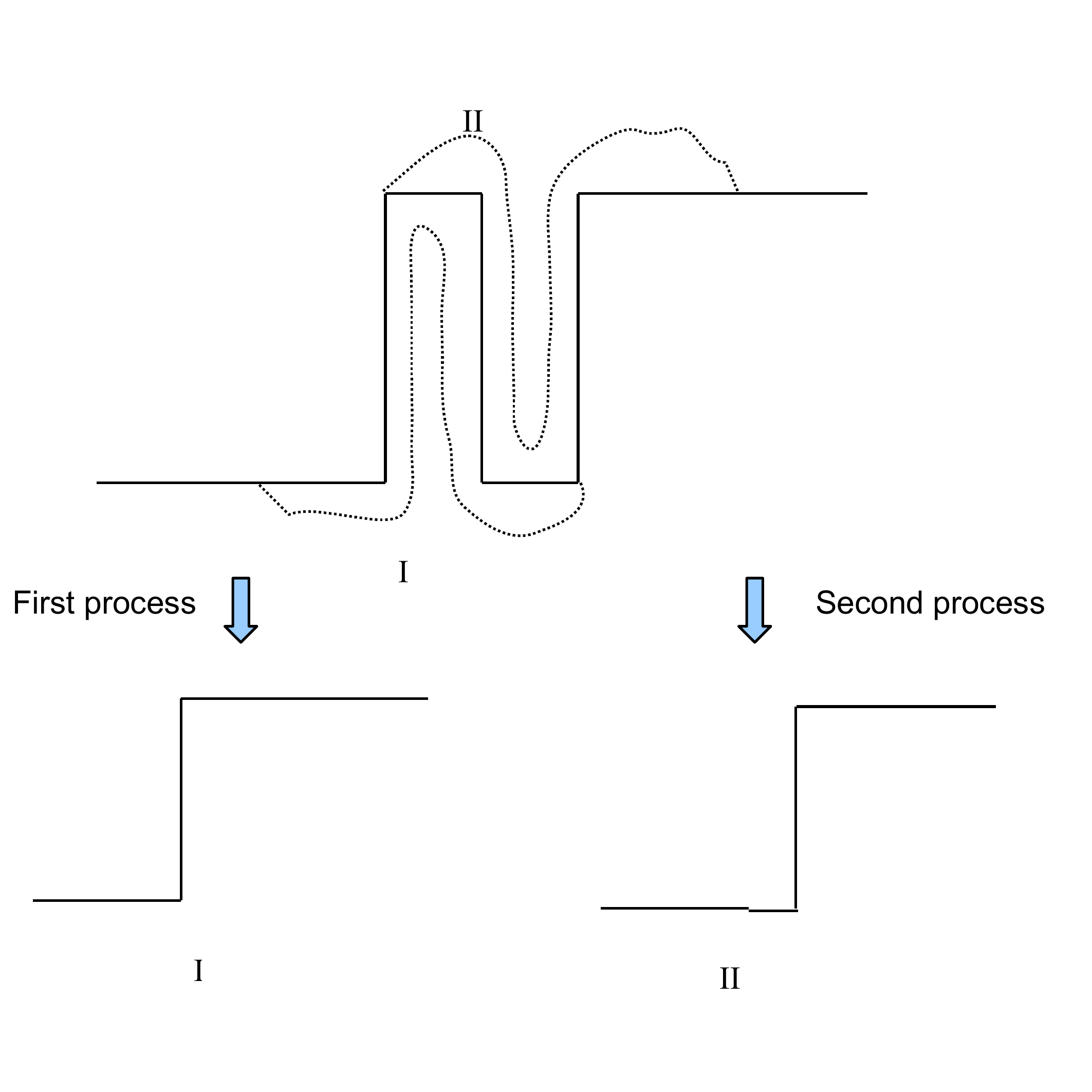}
\caption{Second method}
\label{fig12}
\end{figure}
\end{center}

\subsection{Case of large multiplicity}

\noindent
To clarify the construction let us first take a zero of order 3.\\
If we assume that, $\sum_{i=1}^{k}\varepsilon_{i}<\varepsilon_{0}$ then we can remove the nearly Dirac masses by
building the decreasing vector on the sides (as in figure \ref{fig14}).

\begin{center}
\begin{figure}[H]
\centering
\includegraphics[scale=.4]{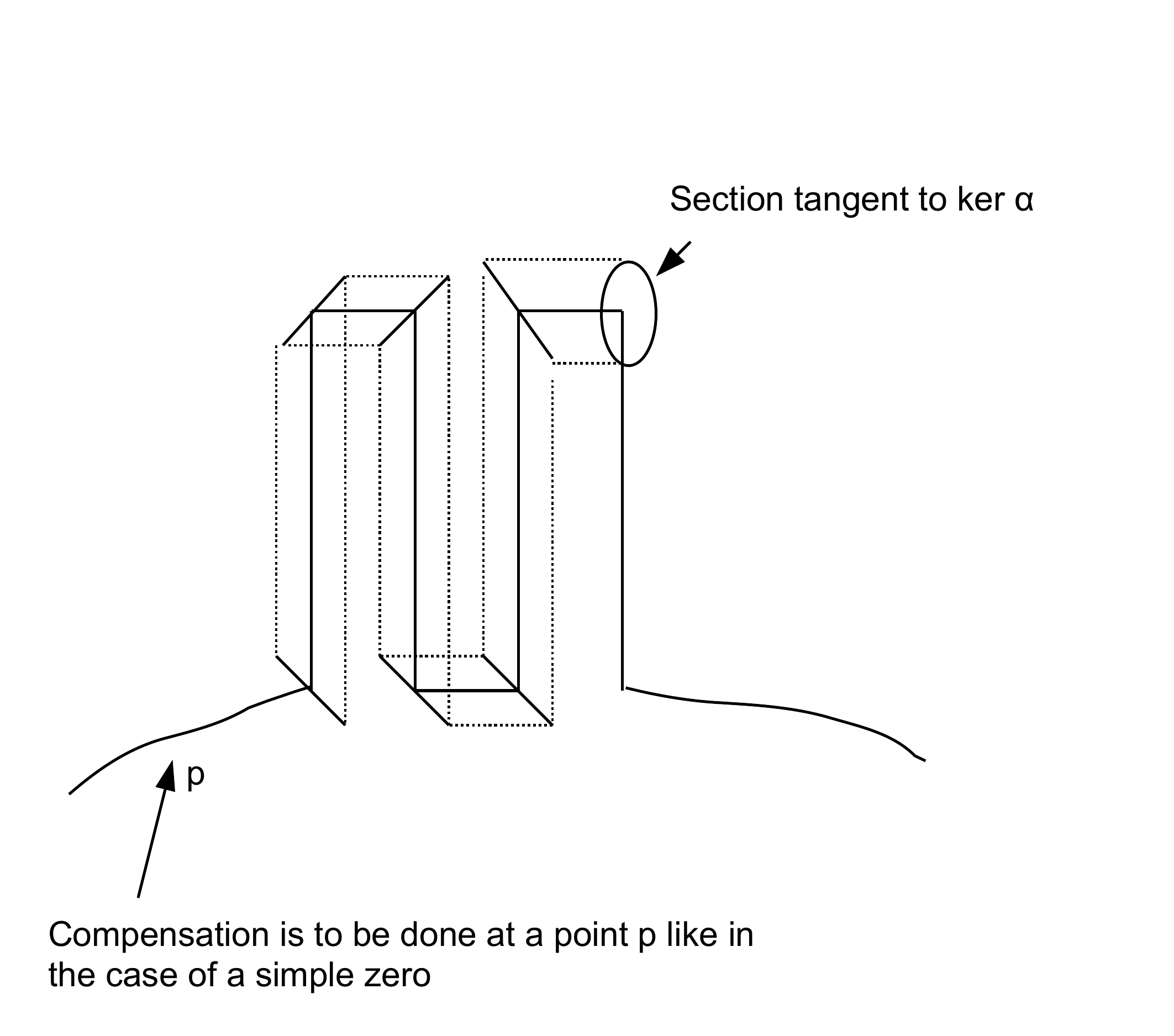}
\caption{The cancellation of multiple Dirac masses}
\label{fig14}
\end{figure}
\end{center}

\noindent
Hence they can be convex combined to lead to a situation of multiple positive $v$ jumps linked by small pieces, as shown in the following figure \ref{fig16}.\\

\begin{center}
\begin{figure}[H]
\centering
\includegraphics[scale=.4]{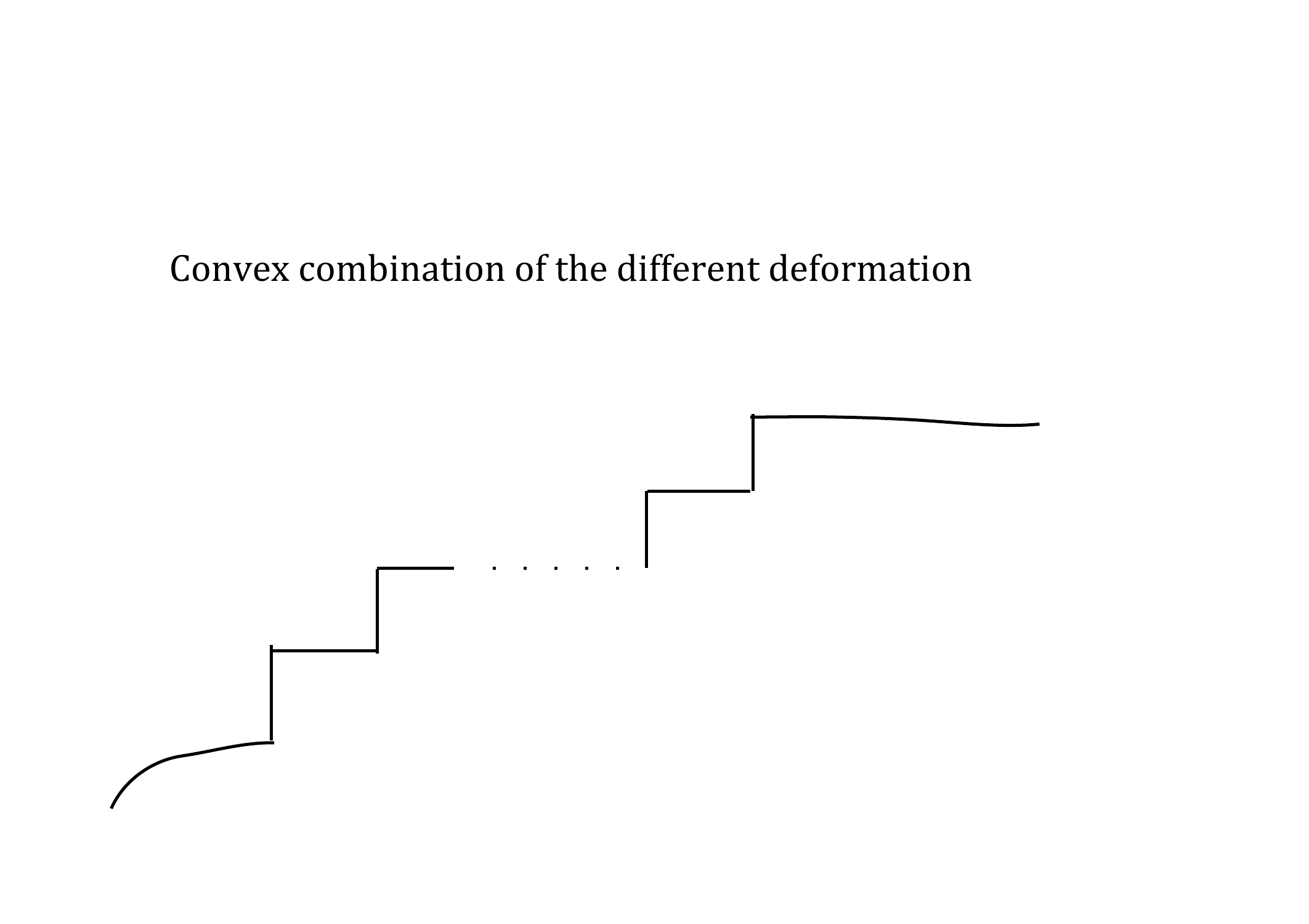}
\caption{Multiple convex combinations.}
\label{fig16}
\end{figure}
\end{center}

\section{``Pushing'' in $\mathcal{C}_{\beta}$}

\noindent
After Cancelling the singularities that appears during the lifting process, we end-up with curves in $\mathcal{C}_{\beta}^{+}$ having consistent pieces with $a>0$. In this section we will proceed with the final step which consists of pushing curves from $\mathcal{C}_{\beta}^{+}$ into $\mathcal{C}_{\beta}$ continuously. This again will be done by the use of a flow that is constructed in a similar way as in section 2. Indeed, we will construct a flow that induces a heat type flow on the component along $\xi$ and by the use of a result of S. Angenent \cite{An}, we will see that after small time, the curves will be deformed into ones in $\mathcal{C}_{\beta}$.

\begin{center}
\begin{figure}[H]
\centering
\includegraphics[scale=.4]{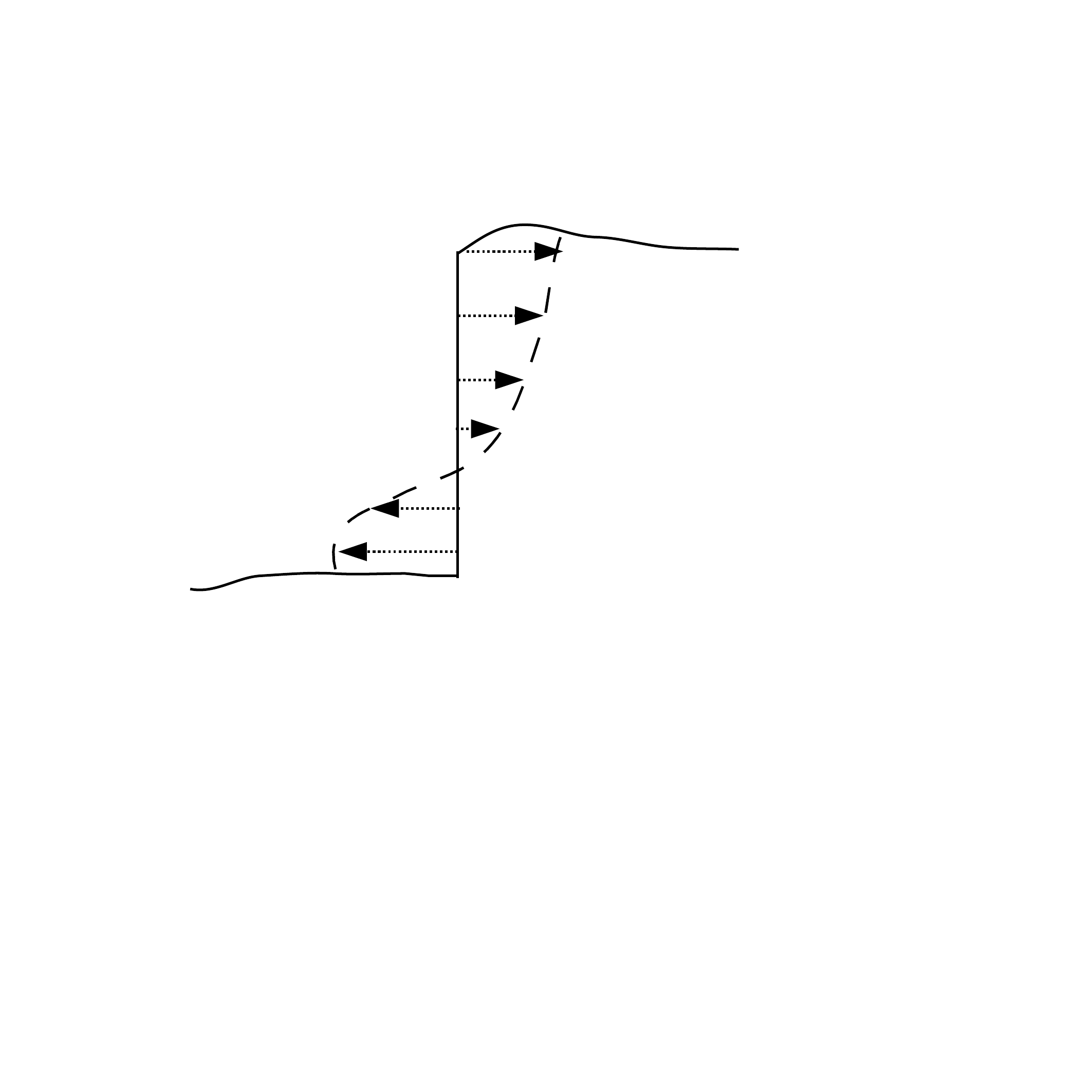}
\caption{Removing the $v$ pieces}
\label{fig15}
\end{figure}
\end{center}

\noindent
First let us recall that we are deforming a curve $x$ in $\mathcal{L}_{\beta}$, that is $\dot{x}=a\xi+bv$, along a vector field $Z=\lambda \xi +\mu v +\eta [\xi,v]$, and we have:
\begin{equation}\label{eq var}
\left\{ \begin{array}{ll}
\displaystyle\frac{\partial a}{\partial s}=\dot{\lambda}-b\eta\\
\\
\displaystyle\frac{\partial b}{\partial s}=\dot{\mu}+(a\tau-\overline{\mu}_{\xi}b)\eta
\end{array}
\right.
\end{equation}

\noindent
We will assume in what follows that $a$ is not identically zero. That is we do not consider periodic orbits of $v$. In fact this can be always assumed after deforming our compact set of curves using the vector field $Z$ constructed in Appendix B.\\
Now we will focus on the first equation of (\ref{eq var}), that is the evolution of $a$. So in this part we take $\lambda=\dot{a}+f$, $\mu=\dot{b}$ and $\eta$ satisfying the usual equation of $\mathcal{L}_{\beta}$, i.e $\dot{\eta}=\overline{\mu}b\eta+\mu a -\lambda b$. Hence: $$\eta=e^{\int_{0}^{t}b(u,s)\overline{\mu}(u)du}
\int_{0}^{t}e^{-\int_{0}^{r}b(u,s)\overline{\mu}(u)du}((a\dot{b}-\dot{a}b)-bf)(r,s)dr.$$
So if we look at the evolution of $a$, we get
$$\frac{\partial a}{\partial s}=\ddot{a}+\dot{f}-b\eta$$
Therefore, if we can find a function $f$ such that $\dot{f}-b\eta=ah$, we can insure the positivity of $a$ starting from a non-negative initial data (as it will be explained later on). But this is equivalent to solving the linear non-homogenous integro-differential equation
$$\dot{f}+be^{\int_{0}^{t}b(u,s)\overline{\mu}(u)du}\int_{0}^{t}e^{-\int_{0}^{r}b(u,s)\overline{\mu}(u)du}bfdr=$$
$$=be^{\int_{0}^{t}b(u,s)\overline{\mu}(u)du}\int_{0}^{t}e^{-\int_{0}^{r}b(u,s)\overline{\mu}(u)du}(a\dot{b}-\dot{a}b) dr+ah$$

\noindent
Notice that we need to find a periodic solution to this equation, so we define the operator $K$ on the space $C_{per}([0,1])$ in the following way: $$K(f)(t)=\int_{0}^{t}b[-\int_{0}^{l}e^{\int_{r}^{l}b(u,s)\overline{\mu}(u)du}bfdr+\int_{0}^{l}e^{\int_{r}^{l}b(u,s)\overline{\mu}(u)du}(a\dot{b}-\dot{a}b) dr]+ah dl$$
Since we want periodicity, we will take $h=c(f)a$ where $c(f)$ satisfies  $$\int_{0}^{1}b[-\int_{0}^{l}e^{\int_{r}^{l}b(u,s)\overline{\mu}(u)du}bfdr+
\int_{0}^{l}e^{\int_{r}^{l}b(u,s)\overline{\mu}(u)du}(a\dot{b}-\dot{a}b)dr] dl=c(f)\int_{0}^{1}a^{2}(l)dl$$
Notice that $c(f)$ is an affine function of $f$, thus $c(f)=c_{1}(f)+c_{2}$. Therefore the final form of the operator $K$ is $$K(f)=\int_{0}^{t}\int_{0}^{l}e^{\int_{r}^{l}b(u,s)\overline{\mu}(u)du}(a\dot{b}-\dot{a}b) dr+c_{2}a^{2}dl +T(f)(t),$$
where $T(f)$ is the bounded linear operator on $C([0,1])$ defined by $$T(f)(t)=\int_{0}^{t}-b\int_{0}^{l}e^{\int_{r}^{l}b(u,s)\overline{\mu}(u)du}bfdr+c_{1}(u)a^{2}dl$$
So the problem now is reduced to find a fixed point for the operator $K$. For that we will use the contraction mapping theorem for an iterate of $K$. The main estimate that is needed reads as
$$\|K^{n}(f_{1})-K^{n}(f_{2})\|\leq \frac{\|T\|^{n}}{n!}\|f_{1}-f_{2}\|$$
where $\|\cdot\|$ stands for the $L^{\infty}$ norm.\\
Thus we have the existence and the uniqueness of $f$ and this leads to the diffusion equation
$$\frac{\partial a}{\partial s}=\ddot{a}+ca^{2}$$
To be more precise about the existence of this flow, one should follow the same procedure as in section 2. That is, we need to regularize the coefficients of the deformation vector to get classical existence, then we need to show that indeed we have convergence to a flow on the curves. Since the procedure is similar to that in section 2, we will omit it.\\
Now we refer to the work of Angenent \cite{An}, about the zeros of parabolic equations of the form
$$\displaystyle\frac{\partial a}{\partial s}=\ddot{a}+g_{1}\dot{a}+g_{2}a$$
We know that the number of zeros of $a$ is non-increasing and if we have $a(s,t_{0})=\dot{a}(s,t_{0})=0$ then the flow will move toward the direction canceling the zero. In our case all the curves in $\mathcal{C}_{\beta}^{+}$ have $a\geq 0$, hence if $a$ is not identically zero then along the flow $a$ will become strictly positive: that is $a(s,t)>0$ for $s>0$.

\section*{Appendix A. Extension of the deformation vector}

\noindent
In this appendix we will see how we can extend the vector field constructed in section 3 to a global deformation on $\mathcal{C}_{\beta}^{+}$.\\
Before we start our extension, let us recall how one can compute the evolution of the tangent to a curve along a deformation vector. We consider here a curve $x\in \mathcal{H}^{2}(S^{1},M)$ such that
$$\dot{x}=a\xi +bv +c [\xi,v]$$
and we also consider a vector field
$$Z=\lambda \xi + \mu v +\eta [\xi,v]$$

\begin{proposition}
Let us assume that $x$ evolves under the flow of $Z$, that is
$$\frac{\partial x}{\partial s}=Z(x),$$
then the following hold:\\

(i) $\displaystyle\frac{\partial a}{\partial s}=\dot{\lambda}-\eta b+\mu c$,\\

(ii) $\displaystyle\frac{\partial b}{\partial s}=\dot{\mu}+\eta (\tau a-\overline{\mu}_{\xi}b)+c(\overline{\mu}_{\xi} \mu-\lambda \tau)$\\

(iii) $\displaystyle\frac{\partial c}{\partial s}=c \mu \overline{\mu}+\dot{\eta}-\overline{\mu}b\eta +\mu a - \lambda b$\\

\end{proposition}

\begin{proof}
(i) Notice that $a=\alpha(\dot{x})$, hence
$$\frac{\partial}{\partial s}a=Z\cdot a =(Z\cdot \alpha)(\dot{x})+\alpha (Z\cdot \dot{x})=$$
$$=d\alpha(Z,\dot{x})+\alpha(\dot{Z})=\dot{\lambda}-\eta b + \mu c$$

\noindent
(ii) We consider the 1-form $\gamma(\cdot)=-d\alpha(\cdot,[\xi,v])$ so we have $b=\gamma(\dot{x})$. Therefore
$$\frac{\partial}{\partial s}b=d\gamma(Z,\dot{x})+\gamma(\dot{Z}).$$
Now
$$d\gamma(Z,\dot{x})=(\lambda b -\mu a)d\gamma(\xi,v)+(\lambda c -\eta a)d\gamma(\xi,[\xi,v])+(\mu c -b \eta)d\gamma(v,[\xi,v]),$$
but
$$d\gamma(\xi,v)=\xi \gamma(v)-v \gamma(\xi)-\gamma([\xi,v])=0,$$
$$d\gamma(\xi,[\xi,v])=-\gamma([\xi,[\xi,v]])=\tau$$
and
$$d\gamma(v,[\xi,v])=-\gamma([v,[\xi,v]])=d\alpha([v,[\xi,v]],[\xi,v])=\overline{\mu}_{\xi}.$$
Thus
$$\frac{\partial}{\partial s}b=\dot{\mu}+\eta (-\overline{\mu}_{\xi}b+\tau a)+c(-\lambda \tau +\overline{\mu}_{\xi} \mu).$$

\noindent
(iii) Here $c=-\beta(\dot{x})$, therefore
$$\frac{\partial}{\partial s}c=-d\beta(Z,\dot{x})-\beta(\dot{Z})=$$
$$=-(\lambda b -\mu a)d\beta(\xi,v)-(\lambda c- \eta a)d\beta(\xi,[\xi,v])-(c\mu-\eta b)d\beta(v,[\xi,v]).$$
A similar computation to the one in (ii) shows that
$$d\beta(\xi,v)=1, \qquad d\beta(\xi,[\xi,v])=0$$
and
$$-d\beta(v,[\xi,v])=d\alpha(v,[v,[\xi,v]])=\overline{\mu}$$
Hence
$$\frac{\partial}{\partial s}c=\dot{\eta}-(\lambda b-\mu a)-\overline{\mu} \eta b+c \mu \overline{\mu}$$
\end{proof}

\noindent
Given a curve $x\in \tilde{K}\cap \mathcal{C}_{\beta}^{+}$, where $\tilde{K}$ is the image of the compact set $K\subset\mathcal{L}_{\beta}$ under the regularizing flow constructed in section 2, such that $\dot{x}=a\xi+bv$. The construction of the vector in section 3 depends on the point $p$ and on the zeros of $a$ at the points B and C as shown in figure (\ref{fig8}). So we write this vector $Z_{p,C,B}$, noticing that the same construction works for $\varphi_{1}$ and $\varphi_{2}$ close to $a$ and $b$ in $\mathcal{H}^{2}(S^{1},M)$ with $p$, $A$ and $B$ the same. Therefore, there exist two neighborhood $U(a)$ and $U(b)$ for which the vector field $Z_{p,C,B}$ is well defined. This constitutes an open cover of the space $\mathcal{H}^{2}(S^{1},M)$ for $a$ and $b$. \\
Now, since $\mathcal{H}^{2}(S^{1},M)$ is paracompact, we can extract a refined cover $(U_{i})_{i\in I}$ that is locally finite and an adapted partition of unity $(\psi_{i})_{i\in I}$. We then use the global deformation $$Z=\sum_{i\in I}\psi_{i}Z_{p_{i},C_{i},B_{i}}.$$
Observe that each $Z_{p_{i},C_{i},B_{i}}$ allows us to compensate our combination of deformations decreasing the ``Dirac mass'' from A to C. The deformation from A to C does not depend on $p_{i}$, $C_{i}$, $B_{i}$. Then $Z$ will also ``compensate'' the deformation. To complete this part, we need to show that indeed $Z$ is a vector field that defines a flow (at least locally). For instance if we can show that $Z_{p,C,B}$, is Lipschitz, then the proof is finished.

\begin{lemma}
Consider a vector field $V \in T\Lambda (M)$, such that for $x\in \Lambda (M)$ $$V(x)=\lambda(x)\xi + \mu(x) v + \eta(x) [\xi,v].$$
Then, if the functions $\lambda$, $\mu$, and $\eta$ are Lipschitz then so is $V$.
\end{lemma}
\begin{proof}

\noindent
Let us fix $x\in \Lambda (M)$, then there exists a neighborhood $U_{x}$ of $x$ in $\Lambda (M)$ such
that for every $\tilde{x}\in U_{x}$, there exists $h \in x^{*}TM$ such that $\tilde{x}(t) = exp_{x(t)}(h(t))$. Hence, this brings the study to curves in $\mathcal{H}^{2}_{loc}(S^1,\mathbb{R}^3)$.\\
The vector field $V$ in this case can be seen as acting on $h$, since $V (\tilde{x})(t) = V (exp_{x(t)}(h(t))$. This yields the regularity of $V$, given the regularity of the coefficients.
\end{proof}

\noindent
We consider now the vector field $Z_{p,C,B}$ constructed on a given curve $x$. This vector contains mainly two parts. The first one is obtained by transporting $-v$, and it depends smoothly on the curve since it depends on the transport equation of the curve. The second part is the one obtained by solving a differential system of the form $\dot{Z}=AZ+H$. If we show that the component of the solution $Z$ has Lipschitz dependence on the curve, then combined with the previous lemma, this proves the result.

\begin{lemma}
The resolvent $R$ of the system satisfying $\dot{R}=AR$, as function of the curve, is Lipschitz.
\end{lemma}

\begin{proof}
The proof is a computational consequence of the formula $$R(\tilde{x})(t)=R(x)(t)+R(x)(t)\int_{0}^{t}R(x)^{-1}(s)(A(x)-A(\tilde{x}))(s)R(\tilde{x})(s)ds.$$
\end{proof}

\noindent
Let us consider now for $x_{0} \in \mathcal{C}_{\beta}^{+}$ the solution to the flow generated by $Z$. That is
$$\left\{ \begin{array}{ll}
\displaystyle\frac{\partial x}{\partial s} =Z(x)\\
\\
x(0)=x_{0},
\end{array}
\right.
$$

\noindent
So for $0<s<s_{0}$, $x(s)$ will be in a certain neighborhood $U$ of $x_{0}$. Hence $Z(x(s))=\sum_{i=1}^{n}\psi_{i}Z_{p_{i},C_{i},B_{i}}$. Thus we have
$$\frac{\partial b}{\partial s}=\sum_{i=1}^{n}\psi_{i} (x(s))(h_{\delta,i}+h_{\delta',i})$$
So we have
$$|b(0,t)-b(s,t)|\leq s (\delta+\delta')$$
Now adapted to the opening $\varepsilon$ and the length $l$ of the nearly Dirac mass we have
$$|b(0,t)-b(s,t)|\leq C \varepsilon l(\delta+\delta')$$

\section*{Appendix B. Perturbation of the periodic orbits of $v$}

\noindent
Let us consider now the periodic orbits of $v$, if there is any. We want to perturb them using the flow of a vector field $Z$ so that they have a part with $a\not = 0$. Let us recall that the variation of $a$ along a vector field $Z=\lambda\xi + \mu v + \eta [\xi,v]$ is given by
$$\frac{\partial a}{\partial s}=\dot{\lambda}-b \eta$$
Hence we want to solve the following system:
$$\left\{ \begin{array}{ll}
\dot{\lambda}-b\eta=h\\
\dot{\eta}=\overline{\mu}b\eta -\lambda b
\end{array}
\right.
$$
For a certain $h\geq 0$. This can be written as
$$\dot{X}=bAX+H$$
where
$$X=\left[ \begin{array}{cc}
\lambda\\
\eta
\end{array}
\right], \quad
A=\left[ \begin{array}{cc}
0 & 1\\
\overline{\mu} & -1
\end{array}
\right], \quad
H=\left[ \begin{array}{cc}
h\\
0
\end{array}
\right]$$
We take a point $p=x(0)$ where $b \not =0$ then it is easy to see that for $t_{0}$ small enough $R(t_{0})-id$ is invertible, where $R(t)$ is the resolvent of the system. This follows from the fact that $R(t)-id=tA(0)+o(t)$, hence $det(R(t)-id)=-bt^{2}+o(t^{2})$.\\

\begin{center}
\includegraphics[scale=.4]{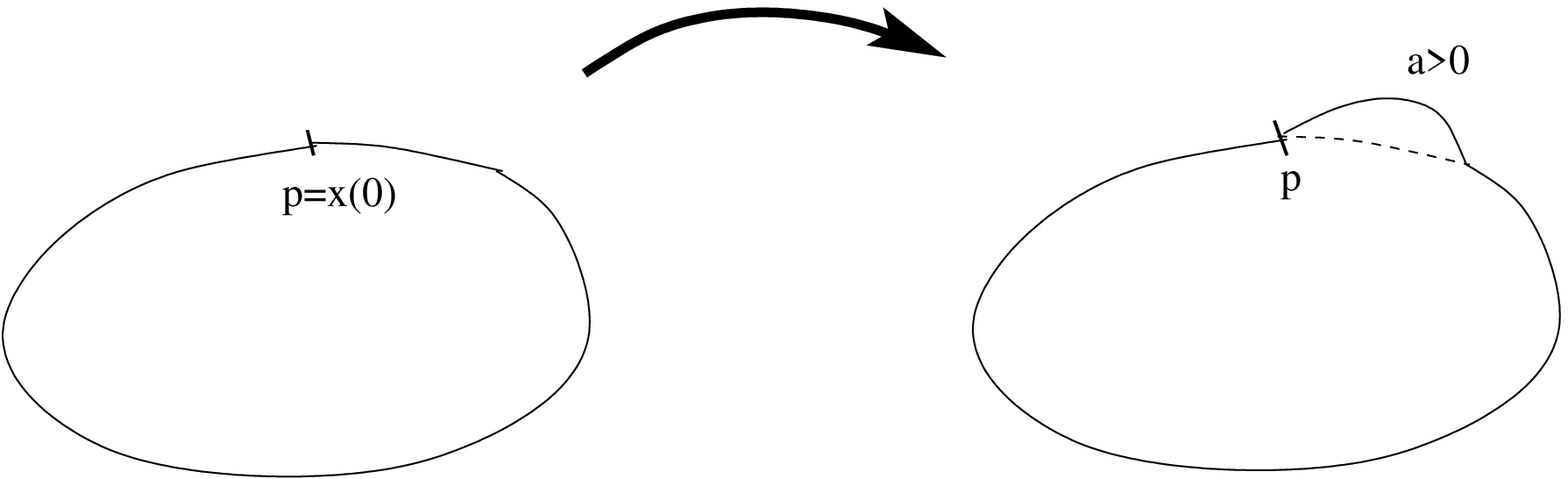}
\end{center}

\noindent
So we take $h$ to be supported in the interval $[0,t_{0}]$. It is important to notice that $Z$ depends on $p$ and $t_{0}$ hence we can write it as $Z_{p,t_{0}}$. Now we need to extend this deformation globally. In a similar way as before we can take an adapted partition of unity $(U_{i}, \psi_{i})$ to the periodic orbits of $v$, so that the vector field $Z$ is globally defined by $Z=\sum _{i} \psi_{i}Z_{p_{i},t_{0,i}}$.

\end{document}